\documentclass[review,hidelinks,onefignum,onetabnum,sort]{siamart250211}
\usepackage{mathrsfs} % Enables \mathscr
\usepackage{amsmath} % Enables \operatorname*
\usepackage{enumerate}
\usepackage{caption}
\usepackage{subcaption}
\usepackage{booktabs} % Enables \toprule, \midrule, \bottomrule
\usepackage{amssymb}

\DeclareMathOperator*{\argmin}{arg\,min} % Defines \argmin

\usepackage{lipsum}
\usepackage{amsfonts}
\usepackage{graphicx}
\usepackage{epstopdf}
\usepackage{algorithmic}
\ifpdf
  \DeclareGraphicsExtensions{.eps,.pdf,.png,.jpg}
\else
  \DeclareGraphicsExtensions{.eps}
\fi

\newsiamremark{remark}{Remark}
\newsiamremark{hypothesis}{Hypothesis}
\crefname{hypothesis}{Hypothesis}{Hypotheses}
\newsiamthm{claim}{Claim}
\newsiamremark{fact}{Fact}
\crefname{fact}{Fact}{Facts}

\headers{MTSM: Multiple tangent space model}{H. Wang, R. Vandebril, J. Van der Veken and N. Vannieuwenhoven}

\title{Manifold-valued function approximation from multiple tangent spaces\thanks{Submitted to the editors DATE.
\funding{This project was funded by BOF project C16/21/002 by the Internal Funds KU Leuven and FWO project G080822N. R. Vandebril is additionally supported by the Research Foundation–Flanders (Belgium), projects G0A9923N and G0B0123N. J. Van der Veken is additionally supported by the Research Foundation---Flanders (FWO) and the Fonds de la Recherche Scientifique (FNRS) under EOS Project G0I2222N. }}}

\author{Hang Wang\thanks{KU Leuven, Department of Computer Science, Celestijnenlaan 200A, B-3001 Leuven, Belgium (\email{hang.wang@kuleuven.be}).} %, \url{http://www.imag.com/\string~ddoe/}).}
\and Raf Vandebril\thanks{KU Leuven, Department of Computer Science, Celestijnenlaan 200A, B-3001 Leuven, Belgium (\email{raf.vandebril@kuleuven.be}).}
\and Joeri Van der Veken\thanks{KU Leuven, Department of Mathematics, Celestijnenlaan 200B, B-3001 Leuven, Belgium (\email{joeri.vanderveken@kuleuven.be}).}
\and Nick Vannieuwenhoven\thanks{KU Leuven, Department of Computer Science, Celestijnenlaan 200A, B-3001 Leuven, Belgium. Leuven.AI, KU Leuven Institute for AI, B-3000 Leuven, Belgium (\email{nick.vannieuwenhoven@kuleuven.be}).} }
\usepackage{amsopn}

\newcommand{\Rr}{\mathbb{R}}
\newcommand{\Mm}{\mathcal{M}}
\newcommand{\LOG}{\operatorname{Log}}

\newcommand{\dist}{\operatorname{dist}}

\begin{document}
\nolinenumbers
\maketitle

% REQUIRED
\begin{abstract}
		Approximating a manifold-valued function from samples of input-output pairs consists of modeling the relationship between an input from a vector space and an output on a Riemannian manifold. We propose a function approximation method that leverages and unifies two prior techniques: (i) approximating a pullback to the tangent space, and (ii) the Riemannian moving least squares method. The core idea of the new scheme is to combine pullbacks to multiple tangent spaces with a weighted Fr\'echet mean. The effectiveness of this approach is illustrated with numerical experiments on model problems from parametric model order reduction.
\end{abstract}
% 
% REQUIRED
% 
% REQUIRED

\section{Introduction} 
	Function approximation involves approximating a function by a simpler function. Traditional methods focus on approximating functions between Euclidean spaces. However, there is growing interest in approximating functions that map from a vector space into a manifold because such \emph{manifold-valued functions} arise naturally in a variety of applications. They include among others \emph{subspace tracking} in signal processing for dynamic direction of arrival estimation and beamforming \cite{MVV1992,Yang1995,Rabideau1996,Delmas_2010} where the function maps into the Grassmannian; \emph{parametric model order reduction} based on (quasi-)interpolation of projections \cite{benner2015pmor,Amsallem2011ipmor,zimmermann2021manifold} where the functions also map into the Grassmannian; \emph{diffusion tensor imaging} in medical imaging applications \cite{Shi09,Lin17} where functions map into the manifold of symmetric positive-definite matrices; \emph{dynamic low-rank approximation} of solutions of partial differential equations \cite{KL2007,NL2008,zimmermann2020herimite,CL2021,CL2023,seguin2024hermite} where the functions map into manifolds of low-rank matrices or tensors; and stiffness matrix prediction for additive manufacturing simulations \cite{DVLM2022} where the function maps into the manifold of low-rank positive semidefinite matrices.
	These applications show that common reasons for approximating a possibly implicit function from samples by an explicit surrogate function include reducing the cost of evaluation, reducing the cost of storing the function, and discovering an explicit description of an implicit function given through samples.
	
	This paper seeks to approximate the---possibly unknown---function
	\begin{equation} \label{eqn_f}
	f: \Rr^n \supset \Omega \longrightarrow \Mm{} \subset \Rr^m, 
	\end{equation}
	where $\Omega$ is open and $\Mm{}$ is an embedded \emph{Riemannian submanifold} of $\Rr^m$, oftentimes equipped with the Euclidean metric from $\Rr^m$ (see \cref{sec_background} below for the background).
	Our aim is finding a function $\widehat{f} : \Omega \to \Mm{}$ that approximates $f$ from $N\in\mathbb{N}$ samples
	\begin{equation} \label{eqn_samples}
	 S := \{ (x_i,y_i) \in \Omega \times \Mm{} \mid i = 1, \dots, N \}.
	\end{equation}
	We make no assumptions on which samples $S$ are provided to us: they could be (i) a fixed set of known input-output pairs as is common in machine learning; (ii) queried by the user from a given, usually expensive-to-evaluate function $f$; or (iii) a combination thereof. The techniques we consider in this paper assume that a set of samples $S$ is given; we do not investigate how to choose the samples if we have this liberty.
	Moreover, for simplicity, we assume that the samples $y_i$ lie exactly on the manifold $\Mm{}$. If this is not the case, then the proposed techniques can still be employed either by (i) projecting the output samples to the manifold $\Mm{}$, or (ii) using the \emph{smooth extension} of the \emph{logarithmic map} from $\Mm{}$ to the ambient $\Rr^m$.
	
	The main challenge of this approximation problem is that the function spaces over $\Omega$ are not linear due to the nonlinearity of $\Mm{}$. This causes difficulty in generalizing classic function approximation methods.
	Nonetheless, there are two main strategies for solving manifold-valued function approximation in the literature: one is to \emph{linearize the problem} locally, as in \cite{Amsallem2011ipmor,simon2024approx,ralf2023approx,zimmermann2021manifold} and \cite[Section 4]{Zimmermann2024multihermi}, while the other is to compute a weighted \emph{Riemannian center of mass}, \emph{Karcher mean} \cite{karcher}, or \emph{Fr\'echet mean} \cite{Frechet1948} of the function values at sample points, as in \cite{Grohs17,sharon2023mrmls,sander2012geo,sander2016geo} and \cite[Section 3]{Zimmermann2024multihermi}.
	
	Linearization constitutes an effective methodology for reducing a manifold-valued function approximation problem to a classic vector-valued function approximation problem by pulling back the approximation problem to a single tangent space, as elucidated in \cite{Amsallem2011ipmor,simon2024approx,zimmermann2020herimite}.
	We refer to this technique as a \emph{single tangent space model} (STSM); it is described in more detail in \cref{sec_1tsm}. For now it suffices to clarify that STSM encompasses two principal stages. The first stage involves pulling back the outputs $\{y_1,\dots,y_N\} \subset \Mm{}$ to a single \emph{tangent space} $T_{p^*}\Mm{}$, where $p^*$ is a designated \emph{anchor point}, which could be, for example, a randomly selected output from $\{y_1,\dots,y_N\}$ or their Fr\'echet mean.
	This pullback usually takes the form of the logarithmic map $\operatorname{Log}_{p^*} : \Mm{} \to T_{p^*} \Mm{}$; see \cref{subsec:Rie_basis}.\footnote{Throughout this paper, we will work with the exponential map and its inverse, primarily so that we can rely on standard concepts from differential geometry. For many matrix manifolds in their natural geometries, the exponential map and its inverse can be efficiently computed \cite{manoptAb,mnoptbo}. However, in the context of Riemannian optimization \cite{manoptAb,mnoptbo} it is common to replace the exponential map by a computationally cheaper approximation thereof, which is called a \emph{retraction}. We believe that most of the results of this paper can be generalized straightforwardly to retractions as well.}
	The second stage entails approximating the function $g: \Rr^n \rightarrow T_{p^*}\Mm{}$ from the pullback samples $\{ (x_i, \mathrm{Log}_{p^*}y_i) \}_{i=1}^N \subset \Mm{} \times T_{p^*}\Mm{}$. This is an approximation problem between vector spaces so standard function approximation schemes can be used, such as multivariate Hermite interpolation as in \cite[Section~4]{Zimmermann2024multihermi} and tensorized Chebyshev interpolation as in \cite{simon2024approx}. STSM is well defined only when all points $\{ y_1, \dots, y_N \}$ lie within the injectivity radius of the point $p^*$.
	
	The approximation models from \cite{Grohs2012,Grohs17,sharon2023mrmls,sander2012geo,Zimmermann2024multihermi,sander2016geo} utilize a weighted Fr\'echet mean as the basis of an approximation of $f$.
	They can be thought of as generalizing the linear \emph{moving least-squares method} \cite{wendland04} that approximates a function $f(x)$ between vector spaces by a linear combination of the outputs as $\widehat{f}(x) := \sum_{i=1}^N \phi_i(x) y_i$, where $\phi_i$ are suitable weight functions, to the setting where the outputs $y_i$ live on a Riemannian manifold. In these methods, proposed originally in \cite{Grohs2012,Grohs17}, the linear combination is substituted by a weighted Fr\'echet mean, resulting in the \emph{Riemannian moving least squares} (RMLS) method.
	Recently, Sharon, Cohen, and Wendland \cite{sharon2023mrmls} proposed a multi-scale extension of RMLS, called MRMLS, which approximates the error of the previous scale with higher resolution samples using RMLS.
	A different extension was proposed by Zimmermann and Bergmann \cite{Zimmermann2024multihermi}; they introduced the barycentric Hermite interpolation method that additionally interpolates the derivatives when computing the weighted Fr\'echet mean.
	
	Comparing STSM and RMLS, we note that the main advantage of STSM is its computational advantage over RMLS for evaluating the approximated function (the \emph{online stage}). The reason is that RMLS requires computing a weighted Fr\'echet mean involving all sample points with non-zero weights. As no closed expression is known in general for the Fr\'echet mean, this involves using specialized approximation \cite{JVV2012} or general Riemannian optimization \cite{manoptAb,mnoptbo} methods, which require evaluating the manifold's exponential map a number of times that is proportional to the number of nonzero weights. STSM requires but one evaluation of the exponential map during the online stage.
	The computational advantage of STSM during the online stage is paid for during the construction of the model (the \emph{offline stage}): RMLS has no set up cost, while STSM needs to pull back all output samples to the designated tangent space $T_{p^*} \Mm{}$ and then solve a traditional multivariate approximation problem there.
	
	\subsection*{Main contribution}
	To address the limitation of locality in STSM and the high computational cost of the online stage of (multilevel) RMLS, we propose a new approach that unifies STSM and RMLS. We call it the \emph{multiple tangent spaces model} (MTSM). Our scheme utilizes multiple STSMs and combines their predictions with a weighted Fr\'echet mean. It can be thought of as replacing dense clusters of output samples by a single STSM. Hereby, MTSM can circumvent the locality restriction of STSM while significantly reducing the online computational cost of employing RMLS. MTSM trades off these advantages for a more costly offline stage.

	The essential ingredients of MTSM are the following:
	\begin{enumerate}
		\item Choose appropriate anchor points $p^*_1,\dots,p^*_R$ on the manifold $\Mm{}$;
		\item pull $f$ back to the tangent spaces $T_{p_j^*} \Mm{}$, $j\in\{1,\dots,R\}$, where possible, and approximate the vector-valued maps $g_j = \mathrm{Log}_{p^*_j} \circ f$ by $\widehat{g}_{j}$ using classic function approximation methods;% \cite{wendland04};
		\item push each $\widehat{g}_{j}$ forward with $\mathrm{Exp}_{p^*_j}$ and compute a weighted Fr\'echet mean of their predictions.
	\end{enumerate}
	\Cref{prop:error_RTSM} presents an error bound of MTSM and \cref{sec:alg} provides the algorithms for its online and offline stages.

	\subsection*{Outline}
	This paper is organized as follows. In \cref{sec_background}, we recall the relevant terminology from differential geometry and basic properties of weighted Fr\'echet means. In \cref{sec:3}, we review the two main manifold-valued function approximation methods in the literature that MTSM generalizes. In \cref{sec:R-TSM}, we present the manifold-valued function approximation framework MTSM and describe its smoothness, error analysis, and required number of tangent spaces. Thereafter, in \cref{sec:alg}, the algorithms to build and evaluate MTSM are proposed. In \cref{sec:Numer_exp}, we present numerical experiments with two basic approximation problems from \cite{sharon2023mrmls} and \cite{Zimmermann2024multihermi}. Moreover, we employ MTSM for parametric model order reduction on the two well-known benchmark problems Anemometer and Microthruster, using psssMOR \cite{psssmor} and Manopt \cite{manopt}. We conclude with a summary of the key findings and an outlook on potential future research in \cref{sec:conc}.
	
	\section{Background and notation}\label{sec_background}
	We give an overview of the required background material and hereby also fix notation that will be used throughout the paper.
	
	\subsection{Riemannian geometry}\label{subsec:Rie_basis}
	In this subsection, we review standard concepts from Riemannian geometry that we use throughout this paper. We refer to \cref{fig:TM1} for a visualization of some of them. Full technical details can be found in~\cite{lee2003smooth,lee2018rie}.

	%% illustration figure
	\begin{figure}[tb]
		\begin{center}
			\includegraphics{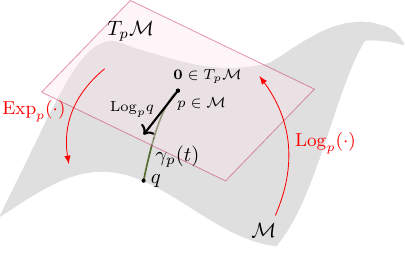}
			\caption{A diagram of key geometric concepts, like the tangent space $T_p\Mm{}$ at $p \in \Mm{}$, the exponential map $\mathrm{Exp}_p$, the logarithmic map $\mathrm{Log}_p$, and a geodesic curve $\gamma_p(t)$.}\label{fig:TM1}
		\end{center}
	\end{figure}

	A $d$-dimensional \emph{topological manifold} is a space where each point has a neighborhood that maps homeomorphically (a continuous bijection with a continuous inverse) to an open subset of $\Rr^d$.
	A \emph{smooth manifold} is, informally, a topological manifold whose transition maps between these various neighborhoods are smooth.
	The \emph{tangent space} $T_p \Mm{}$ at a point $p \in \Mm{}$ is the $d$-dimensional vector space of velocity vectors of smooth curves passing through $p$.
	
	A \emph{Riemannian manifold} $(\Mm{}, \mathbf{g})$ is a smooth manifold that is equipped with a Riemannian metric $\mathbf{g}$. Throughout this paper, we assume a metric has been fixed, and for brevity we refer to $\Mm{}$ as a Riemannian manifold.
	A \emph{Riemannian metric} $\mathbf{g}$ provides a smoothly varying inner product $\mathbf{g}_p: T_p \Mm{} \times T_p \Mm{} \rightarrow \Rr$ on the tangent space at $p \in\Mm{}$. The norm induced by this inner product at $p$ is denoted by $\| \cdot \|_p$.\footnote{This notation is not to be confused with a common notation for the $p$-norm, which is nowhere used in this paper.}
	The \emph{geodesic distance} between two points $p, q \in \Mm{}$ is
	\[
	\operatorname{dist}_\Mm{}(p,q) := \operatorname{inf}\{L(\gamma) \mid \gamma: [0,1] \rightarrow \Mm{} \ \text{piecewise smooth},\gamma(0)=p, \gamma(1)=q\},
	\]
	where $L(\gamma) := \int_{0}^{1} \| \gamma'(t) \|_{\gamma(t)} \mathrm{d}t$ is the \emph{length} of the curve $\gamma$ and $\gamma'(t) \in T_{\gamma(t)} \Mm{}$ is the derivative of $\gamma$ at $t \in (0,1)$. If no curves exist between $p$ and $q$, then $\dist_{\Mm}(p,q)=\infty$.
	
	A \emph{geodesic curve} is a curve $\gamma(t): \Rr \supset [0,1] \rightarrow \Mm{}$ that locally minimizes the length of the path between its endpoints, i.e., there exists a supremum $a'\in (0,1]$ such that $\gamma|_{[0,a']}$ is a \emph{length-minimizing geodesic} in the sense that $\operatorname{dist}_{\Mm{}}(\gamma(0), \gamma(t)) = L(\gamma|_{[0,t]})$ for all $0 \le t \le a'$.
	Geodesic curves starting in $p \in \Mm{}$ can be parameterized by a tangent vector in the tangent space at $p$ via the \emph{exponential map} $\mathrm{Exp}_{p} : D \rightarrow \Mm{}$, for some open $D \subset T_{p} \Mm{}$. It takes a tangent vector $v \in D$ at $p\in\Mm$ and maps it to $\gamma(\|v\|_p)$, where $\gamma$ is the unit-speed geodesic with $\gamma(0)=p$ and $\gamma'(0)=t$.
	
	The \emph{injectivity radius} $\operatorname{inj}(p) > 0$ at $p \in \Mm{}$ is the supremum over $\sigma \in \Rr$ such that
	\[
	\mathrm{Exp}_p : B_{\sigma}(p) \longrightarrow \mathcal{B}_{\sigma}(p) \subset \Mm{}
	\]
	is a diffeomorphism (a smooth bijective map with a smooth inverse), where
	\[
	B_{\sigma}(p) := \{ v \in T_p\Mm{} \mid \| v\|_p < \sigma \}
	\quad\text{and}\quad
	\mathcal{B}_{\sigma}(p) := \mathrm{Exp}_p( B_{\sigma}(p) )
	\]
	are the balls of radius $\sigma$ around, respectively, $0 \in T_p \Mm{}$ and $p \in \Mm{}$.
	If $\sigma \le \operatorname{inj}(p)$, then $\mathrm{Log}_{p} := \mathrm{Exp}_p|_{B_\sigma(p)}^{-1}$ is called the \emph{logarithmic map} and both $B_\sigma(p)$ and $\mathcal{B}_\sigma(p)$ are called \emph{geodesic} balls.
	The logarithmic map takes $q \in \mathcal{B}_\sigma(p)$ to the tangent vector at $p$ of the unique geodesic curve connecting $p$ and $q$.
	
	The curvature of a Riemannian manifold can be quantified as detailed in \cite{gallot1990rie}. 
	Let $\Sigma_p$ be a $2$-dimensional subspace of $T_p\Mm{}$, let $C_{\Sigma_p}(r)$ denote the image under the exponential map at $p$ of the unit circle in $\Sigma_p$ for sufficiently small $r>0$, and let $\ell_{\Sigma_p}(r)$ be the length of $C_{\Sigma_p}(r)$. Then, the \emph{sectional curvature} $\mathrm{sec}(\Sigma_p) \in \Rr$ can be defined as
	\(
	 \mathrm{sec}(\Sigma_p) = \left.\frac{1}{2\pi} \frac{\mathrm{d}^3}{\mathrm{d}r^3} \right\vert_{r=0} \ell_{\Sigma_p}(r).
	\)
	
	We will frequently require suitable sectional curvature bounds in this paper. For this reason, we introduce the next notations for the ``clipped'' infimum and supremum of the sectional curvatures on a manifold $\Mm$:
	\[   
	 \kappa_\ell^-(\Mm) := \min\Biggl\{0, \inf_{\substack{p \in \Mm,\\ \Sigma_p \subset T_p \Mm}} \mathrm{sec}(\Sigma_p) \Biggr\}
	 \text{ and }
	 \kappa_u^+(\Mm) := \max\Biggl\{0, \sup_{\substack{p \in \Mm,\\ \Sigma_p \subset T_p \Mm}} \mathrm{sec}(\Sigma_p) \Biggr\},
	\]
	where $\Sigma_p$ is a $2$-dimensional linear subspace of $T_p \Mm$. This definition ensures that $\kappa_\ell^-(\Mm) \le 0 \le \kappa_u^+(\Mm)$.

	\subsection{Fr\'echet mean}
	In this subsection, we recall the definition and properties of a Fr\'echet mean, which forms the basis of the RMLS method.
	
	The $d$-dimensional \emph{standard simplex} is
	\[
	\Delta^d=\left\{ (w_1,\dots,w_{d+1}) \in \Rr^{d+1} \mid  0 \leq w_1,\ldots,w_{d+1} \le 1 \ \text{and} \ \sum_{i=1}^{d+1} w_i=1\right\},
	\]
	and, in a slight abuse of notation, we will use $\Phi \in \Delta^d$ also for functions $\Phi: \Rr \rightarrow \Delta^d$.
	For weights $\Phi=(\varphi_1,\dots,\varphi_d) \in \Delta^{d-1}$ and data points $Y=(y_1,\dots,y_{d}) \in \Mm{}^{\times d} := \Mm{}\times\dots\times\Mm{}$, the \emph{Fr\'echet mean} \cite{karcher,Frechet1948} is
		\begin{equation}\label{def_karcher_mean}
			\mathrm{avg}_{\Mm{}}: \Mm{}^{\times d} \times \Delta^{d-1} \longrightarrow \Mm{}, \qquad
			( Y,\Phi) \longmapsto \underset{p \in \Mm{}}{\operatorname{argmin}} \sum_{i=1}^{d} \varphi_i \cdot \dist_\Mm{}\left(p, y_i\right)^2.
		\end{equation}
	
	There could be multiple solutions of this minimization problem. However, under the following sufficient condition, a unique minimizer exists.

	\begin{theorem}[Uniqueness \protect{\cite[Theorem 2.1 with $p=2$ and Remark 2.5]{Afsari2011}}]\label{thm:well_karcher}
		If there exists a point $p \in \Mm$ and a radius 
		\[      
		 \sigma \le \frac{1}{2} \min\Biggl\{ \mathrm{inj}( \mathcal{N}_\sigma ), \frac{\pi}{\sqrt{\kappa_u^+( \mathcal{N}_\sigma )}} \Biggr\},
		\]
		with $\mathcal{N}_\sigma$ equal to $\Mm$ or $\mathcal{B}_{2\sigma}(p)$,
		such that $y_1, \dots, y_{d} \in \mathcal{B}_{\sigma}(p)$, then the minimization in the definition of $\mathrm{avg}_{\Mm{}}$ in \cref{def_karcher_mean} has a globally unique solution for all $\Phi \in \Delta^{d-1}$.
	\end{theorem}

	%----------
\section{Prior models for manifold-valued function approximation}\label{sec:3}
	Before introducing MTSM, we briefly review the two manifold-valued function approximation methods from the literature that our method will extend and unify.
	
	\subsection{STSM: Single tangent space model} \label{sec_1tsm}
    This approach from \cite{simon2024approx,zimmermann2021manifold} and \cite[Section 4]{Zimmermann2024multihermi} constructs an approximating function for \cref{eqn_f} from input-output samples as in \cref{eqn_samples} by seeking to solve
	\begin{equation} \label{eqn_stsm}
		\mathscr{S} f := \argmin_{\widehat{f} \colon \Omega \to \Mm{}} \; \sum_{i=1}^{N} \dist_\Mm{}(\widehat{f}(x_i),y_i )^2,
	\end{equation}
	where the approximating function $\widehat{f}$ is anchored at a single anchor point $p^* \in \Mm{}$ and has
	the form
	\begin{equation*}
		\widehat{f}(x):= \mathrm{Exp}_{p^*} \bigl( \widehat{g}(x) \bigr),
	\end{equation*}
	in which $\widehat{g}: \Rr^n \rightarrow T_{p^*}\Mm{}$ is a vector-valued function from a suitable function space.

	The anchor point $p^*$ is often chosen as a Fr\'echet mean, as in \cite{simon2024approx,Zimmermann2024multihermi}, or as a random output sample from $\{y_1,\dots,y_N\}$, as in \cite{Amsallem2011ipmor,zimmermann2021manifold}.

	To determine the vector-valued function $\widehat{g}$, the following auxiliary classic multivariate function approximation problem is solved:
	\begin{equation}\label{eq:TSM_p2}
		\min_{\widehat{g} \colon \Rr^n \to T_{p^*}\Mm{}}\; \sum_{i=1}^{N} \dist_{T_{p^*}\Mm{}}\bigl( \widehat{g}(x_i),\mathrm{Log}_{p^*}(y_i) \bigr)^2,
	\end{equation}
	where the minimization is over the function space determined by the chosen multivariate approximation method and $\dist_{T_{p^*}\Mm{}}$ is the straight-line distance induced by the inner product $\mathbf{g}_{p^*}$ defined by the Riemannian metric $\mathbf{g}$ at $p^*$.
	Any vector-valued function approximation method can be used to solve~\cref{eq:TSM_p2}, because the codomain of $\widehat{g}$ is a vector space of dimension equal to the dimension of $\Mm{}$.
	
	Optimization problem \cref{eq:TSM_p2} is in general not equivalent to \cref{eqn_stsm} because $\mathrm{Exp}_{p^*}$ is only a \emph{radial isometry} \cite{lee2018rie}, which preserves the distances between points $v_1$ and $v_2$ in $T_{p^*}\Mm{}$ only if the straight line passing through $v_1$ and $v_2$ passes through the origin. The distance between any other pair of points will be distorted by the curvature of the Riemannian manifold $\Mm{}$. For this reason, to measure the accuracy of STSM, an error bound of this model was proposed in \cite{simon2024approx}, which is recalled next.

	\begin{theorem}[Error bound of STSM {\cite[Theorem~3.1]{simon2024approx}}]\label{thm_error_TSM}
		Consider a Riemannian manifold $\Mm{}$ and a point $p^* \in \Mm{}$.
		Let $f: \Omega \rightarrow \Mm{}$, where $\Omega$ is a set whose image is contained in a geodesic ball $\mathcal{B}_\sigma(p^*) \subset \Mm{}$.
		Assume that $\widehat{f}=\mathrm{Exp}_{p^*}\circ \widehat{g},$ where $\widehat{g}: \Omega \rightarrow B_\sigma(p^*)$ is an approximation such that
		\[
		\bigl\|\mathrm{Log}_{p^*}\bigl( f(x) \bigr) -\widehat{g}(x) \bigr\|_{p^*} \leq \epsilon
		\]
		for all $x \in \Omega$. Let $L:=\kappa_\ell^-(\mathcal{B}_\sigma(p^*))$. Then, for all $x \in \Omega$, the distance between $f(x)$ and $\widehat{f}(x)$ on $\Mm{}$ obeys
		\[
		\dist_{\Mm{}}(f(x), \widehat{f}(x)) \leq
		\epsilon +
		\begin{cases}
			 0, & \text { if } L = 0, \\
			 \frac{2}{\sqrt{|L|}} \operatorname{arcsinh} \left(\frac{\epsilon \sinh (\sigma \sqrt{|L|})}{2 \sigma}\right), & \text { if } L<0.
		\end{cases}
		\]
	\end{theorem}

	\Cref{thm_error_TSM} indicates that STSM can approximate a manifold-valued function with high accuracy if the image of $f$ is contained in a small geodesic ball in $\Mm{}$ and the vector-valued pullback function $\mathrm{Log}_{p^*} \circ f$ can be well approximated on $\Omega$. For a smaller negative sectional curvature, a smaller geodesic ball is required to maintain the same error bound.
	
	For the sake of simplicity, we state a more concise corollary of \cref{thm_error_TSM}.
	\begin{corollary}\label{col:error_TSM}
		Under the assumptions of \cref{thm_error_TSM}, if additionally $\sigma \sqrt{|L|} < \pi$, then
		\(
			\dist_{\Mm{}}(f(x), \widehat{f}(x)) \leq 5 \epsilon.
		\)
	\end{corollary}
	\begin{proof}
	If $L = 0$, the bound holds trivially by \cref{thm_error_TSM}. So consider the case $L < 0$.
	Exploiting that $\mathrm{arcsinh}(x) < x$ when $x>0$, we obtain that
		\begin{align*}
			\frac{2}{\sqrt{|L|}} \operatorname{arcsinh}\left(\frac{\epsilon}{2} \cdot \frac{\sinh (\sigma\sqrt{|L|})}{\sigma}\right)
% 			<\frac{2}{\sqrt{|L|}} \left(\frac{\epsilon}{2} \cdot \frac{\sinh (\sqrt{|L|} \sigma)}{\sigma}\right)
			< \epsilon \cdot \frac{\sinh (\sigma \sqrt{|L|})}{\sigma \sqrt{|L|} }.
		\end{align*}
		If $\sigma \sqrt{|L|} < \pi$, we have $\frac{\sinh (\sigma \sqrt{|L|})}{\sigma \sqrt{|L|}}<4$, so that
		$\dist_{\Mm{}}(f(x),\widehat{f}(x))<5\epsilon$.
	\end{proof}

	\subsection{RMLS: Riemannian moving least squares} \label{sec_rmls}
	Aimed at solving a general manifold-valued function approximation problem, the Riemannian moving least squares method, which is based on the Fr\'echet mean and the moving least square method, was proposed by Grohs, Sprecher, and Yu \cite{Grohs17}.

	Recall from \cite{wendland04} that the Euclidean moving least squares method approximates a function $f : \Rr^n \to \Rr$ from input-output samples $\{ (x_i, y_i) \}_{i=1}^N \subset \Rr^n \times \Rr$ as
	\[
	(\mathscr{E} f)(x) :=
	\min_{p \in \pi_m\left(\Rr^n\right)} \;\sum_{i=1}^N \Phi\left(x, x_i\right)\cdot\left( p\left(x_i\right) - f\left(x_i\right)\right)^2,
	\]
	where $\pi_m (\Rr^n )$ denotes the space of all $n$-variate polynomials of total degree less than or equal to $m \in \mathbb{N}$, and the weight function $\Phi(x,x_i)$ calculates weights based on the distance between $x$ and the data points $x_i$. Usually, $\Phi$ is compactly supported so that the weight $\Phi(x,x_j) = 0$ if the distance between $x$ and $x_j$ exceeds a threshold.

	The foregoing scheme was extended to approximate manifold-valued functions as in \cref{eqn_f} from samples as in \cref{eqn_samples} by \cite{Grohs17}.
	If $\Phi : \Rr^n \to \Delta^{N-1}$ is a weight function, then the Riemannian moving least squares approximation is essentially a Fr\'echet mean (cf.~\cref{def_karcher_mean}) with adaptively determined weights:
	\begin{equation} \label{eqn_rmls}
		(\mathscr{R} f)(x):= \argmin_{p \in \Mm{}} \;\sum_{i=1}^{N} \varphi_i(x) \cdot \dist_{\Mm{}}(p,y_i)^2,
	\end{equation}
	where $\varphi_i: \Rr^n \to \Rr$ is the $i$th component of the weight function $\Phi$. The component weight functions can be chosen in many ways, such as Lagrangians as in \cite{sander2016geo} or compactly supported basis functions as in \cite{wendland04}.

	RMLS is well defined only on the domain where the Fr\'echet mean is unique. If the points are contained in a sufficiently small ball, this is guaranteed by \cref{thm:well_karcher}.

	\section{MTSM: Multiple tangent spaces model}\label{sec:R-TSM}
	This section introduces our approach to approximate manifold-valued functions via multiple tangent spaces.
	
		On the one hand, STSM requires that all data points lie in a single cluster on the manifold, enabling efficient function approximation from a local linearization. On the other hand, RMLS can handle data points spread across the manifold so long as the points with nonzero weights are all in a geodesic ball, but it is computationally expensive to evaluate as it involves a weighted Fr\'echet mean.
	
		The essence of the proposed MTSM is to plug in STSMs to replace dense clusters of points on the manifold to reduce the cost of computing a weighted Fr\'echet mean in RMLS. See \cref{fig:M1} for a visualization of an MTSM consisting of three STSMs.
		
		\begin{figure}\centering
			\includegraphics[width=0.8\textwidth]{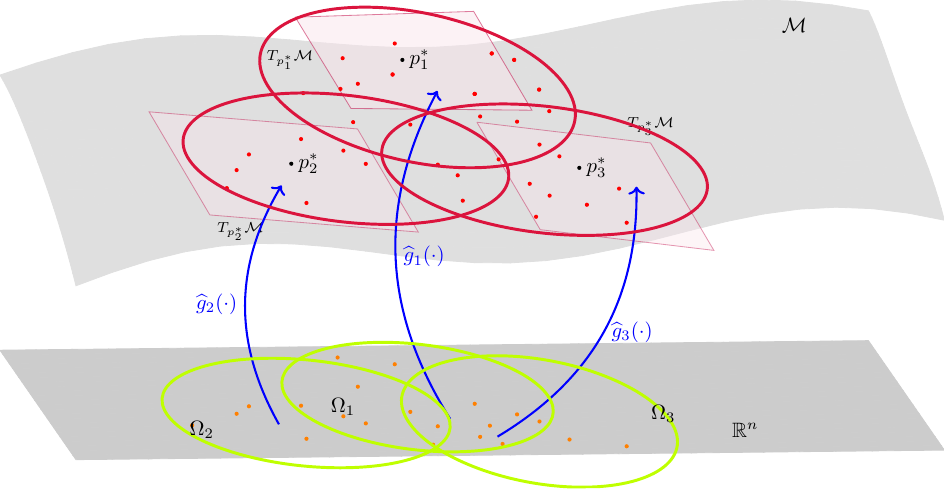}
				\caption{An illustration of a $3$-MTSM. It shows three tangent spaces at the anchor points $p^*_1,\ p^*_2,\ p^*_3$ and the relevant vector-valued functions $\widehat{g}_{j}: \Omega_j \rightarrow T_{p^*_j}\Mm{}$. The approximation consists of computing the weighted Fr\'echet mean of $\mathrm{Exp}_{p^*_j}\widehat{g}_{j}(x)$.}\label{fig:M1}
		\end{figure}
	
		\begin{definition}[Multiple tangent space model]\label{def_mtsm}
		The $R$-MTSM approximates an $f$ as in \cref{eqn_f} from the samples \cref{eqn_samples} by utilizing $R$ individual STSMs, as follows:
		\begin{equation}\label{eqn_mtsm}
			(\mathscr{M} f)(x)
	% 		= \widehat{f}_R(x)
			:= \argmin_{p \in \Mm{}}\sum_{j=1}^{R} \varphi_j\bigl( \widehat{f}_{j}(x) \bigr) \cdot \dist_\Mm{}\bigl( p, \widehat{f}_j(x) \bigr)^2,\quad
			\widehat{f}_j := \mathrm{Exp}_{p^*_j} \circ \widehat{g}_{j},
		\end{equation}
		where, for $j=1,\ldots,R$, we have that
		\begin{enumerate}
		\item $p^*_j \in \Mm{}$ is the anchor point of the $j$th STSM,
		\item $\widehat{g}_{j}: \Omega \supset \Omega_j \rightarrow T_{p^*_j}\Mm{}$ is the vector-valued function approximation of the $j$th STSM, and
		\item $\varphi_j: \Mm{} \rightarrow \Rr$ is the weight function that determines to what extent the $j$th STSM is active in the total approximation of $f$ at $x\in\Omega$.
		\end{enumerate}
		We use as convention that $\varphi_j( \widehat{f}_j(x) ) := 0$ if $x \not\in \Omega_j$.
		\end{definition}
	
	Note that $R$-MTSM can be viewed as a unification of STSM and RMLS. On the one hand, if $R=1$ and the single weight function $\varphi_1(x)$ is taken to be $1$, then this MTSM is an STSM. On the other hand, if $R = N$ and for $j = 1,\ldots,N$ we choose $\widehat{g}_j(x) = 0$ and $p_j^* = y_j$, then MTSM is an instance of RMLS in which all weights $\phi_j(0)$ are constants---naturally, this is not a recommended choice of weights.

	In the next subsections, we investigate some theoretical considerations about MTSM. In \cref{sec:alg}, we propose concrete algorithms to set up and evaluate MTSMs.

		\subsection{Well-posedness}
		How the three components in \cref{def_mtsm} should be chosen so that the minimization problem in the definition of MTSM is well posed is discussed in the next proposition. Practical choices are discussed in \cref{sec:alg}.
	
		\begin{proposition}[Well-posedness]\label{thm_rtsm_unique}
			Consider the map $f : \Omega \to \Mm{}$ from \cref{eqn_f} and the $R$-MTSM from \cref{def_mtsm}.
			Let $L := \kappa_\ell^-(\Mm) \le 0 \le \kappa_u^+(\Mm) =: K$ be the clipped curvature bounds.
			If, for all $j=1,\dots,R$, there exist $0 < \sigma_j \le \tau_j \le \mathrm{inj}(p_j^*)$ with $\tau_j \sqrt{|L|} \le \pi$ such that
			\begin{enumerate}
			\item[A1.] the $\mathcal{B}_{\sigma_j}(p_j^*)$'s cover $f(\Omega) \subset \Mm{}$;
			\item[A2.] the maps $\widehat{g}_j : \Omega_j \to \mathcal{B}_{\tau_j}(p_j^*)$, where $\Omega_j := f^{-1}( \mathcal{B}_{\sigma_j}(p_j^*) )$ is the preimage;
			\item[A3.] for all $x \in \Omega_j$ the error bound
			\[
			\| \widehat{g}_j(x) - (\mathrm{Log}_{p_j^*} \circ f)(x) \|_{p_j^*} \le \epsilon < \frac{1}{10} \min\left\{ \mathrm{inj}( \Mm{} ), \frac{\pi}{2\sqrt{K}} \right\}
			\]
			holds; and
			\item[A4.] the weights $\varphi_j$ form a partition of unity subordinate to the cover of $f(\Omega)$ by the larger geodesic balls $\mathcal{B}_{\tau_j}(p_j^*)$'s;
			\end{enumerate}
			then \cref{eqn_mtsm} is well posed, has a unique global minimizer for each $x \in \Omega$, and MTSM in \cref{def_mtsm} is well defined.
		\end{proposition}
		\begin{proof}If the weight $\varphi_i$ of a point $y_i$ is zero in \cref{def_karcher_mean}, then the Fr\'echet mean with or without the $i$th summand will be the same. Consequently, to determine if \cref{eqn_mtsm} is well posed, it suffices to show that
		\begin{enumerate}
		 \item[C1.] $\varphi_j(\widehat{f}_j(x)) > 0$ only if $x \in \Omega_j$ and $\widehat{g}_j(x)$ lies in the domain of $\mathrm{Exp}_{p_j^*}$;
		 \item[C2.] for all $x \in \Omega$, the Fr\'echet mean is taken over at least $1$ point; and
		 \item[C3.] the $\widehat{f}_j(x)$'s with $\varphi_j(\widehat{f}_j(x)) > 0$ satisfy the conditions of \cref{thm:well_karcher}.
		\end{enumerate}
		We do this in the next paragraphs.
	
		\textit{C1}: Because of the definition of a partition of unity in A4 (see \cite[p.~43]{lee2003smooth}), we observe that $\varphi_j(\widehat{f}_j(x)) > 0$ only if $x \in \Omega_j$ (by A1, A2, and the convention in the definition of $\varphi_j$) and $\widehat{f}_j(x) \in \mathcal{B}_{\tau_j}(p_j^*)$ (which holds by A2). As each $\mathcal{B}_{\tau_j}(p_j^*)$ is an open geodesic ball, $\widehat{g}_j(x)$ lies in the domain of $\mathrm{Exp}_{p_j^*}$.
	
		\textit{C2}: Let $x \in \Omega$ be arbitrary. Since the $\mathcal{B}_{\sigma_j}(p_j^*)$'s cover $f(\Omega)$ by A1, we can assume, after relabeling, that $x \in (\Omega_1 \cap \cdots \cap \Omega_\ell)$ for some $\ell \ge 1$. By A4, the $\varphi_j$'s form a partition of unity subordinate to the cover by the radius $\tau_j$ balls, so that $\varphi_j( \widehat{f}_j(x) ) \ge 0$ and $\sum_{j=1}^\ell \varphi_j( \widehat{f}_j(x) ) = 1$ \cite[p.~43]{lee2003smooth}. Hence, at least $1$ of these weights must be nonzero. Note that by C1 no other weights can be nonzero.
	
		\textit{C3}: Assume again that the nonzero weights are those corresponding to $j=1,\dots,\ell$. By C1, $\widehat{f}_j(x)$ is well defined. Then, exploiting the bound in A3, we conclude from \cref{col:error_TSM} that
		\(
		 \dist_\Mm{}( f(x), \widehat{f}_j(x) ) \le 5 \epsilon < \frac{1}{2} \min\{\mathrm{inj}(\Mm{}), \frac{\pi}{2} K^{-1/2} \}.
		\)
		This implies that the sufficient condition of \cref{thm:well_karcher} holds for $p=f(x)$, so the Fr\'echet mean is unique. This concludes the proof.
		\end{proof}
	
		\begin{remark}
		We can replace $\Mm{}$ by the covering of the larger geodesic balls $\mathcal{B}_{\tau_j}(p_j^*)$ in \cref{thm_rtsm_unique}. This can potentially improve the curvature bounds $L$ and $K$ as well as the injectivity radius.
		\end{remark}
	
		\Cref{thm_rtsm_unique} is not a vacuous statement. Assumptions A1 and A4 can always be satisfied with a finite $R$ for manifolds with a strictly positive injectivity radius. A simple choice of weight functions satisfying A4 is presented in \cref{eqn_specific_weight_function} in~\cref{sec:Numer_exp}. Assumptions A2 and A3 are satisfied for $\widehat{g}_j = (\LOG_{p_j^*} \circ f)|_{\Omega_j}$.
		
		\subsection{Smoothness}
		The key feature of MTSM is that it stitches together predictions from local STSMs through the Fr\'echet mean to obtain a potentially global approximation of a map into a manifold. Our main result is that the patches are stitched together smoothly if the vector-valued approximations are smooth.
	
		\begin{theorem}\label{thm_mtsm_smooth}
		 Under the assumptions of \cref{thm_rtsm_unique}, the MTSM approximation $(\mathscr{M} f) : \Omega \to \mathcal{M}$ defined in \cref{eqn_mtsm} is a smooth map if all $\widehat{g}_j : \Omega_j \to T_{p_j^*} \mathcal{M}$, $j=1,\dots,R$, are smooth maps.
		\end{theorem}
	
		The idea of the proof is relatively straightforward: Show that \cref{eqn_mtsm} is the map that takes its input to the point on $\mathcal{M}$ where the gradient of the function inside of the minimization vanishes and then use the \emph{implicit function theorem} to show that such an implicitly defined function is smooth. Working out the technical details, however, requires more advanced machinery from Riemannian geometry. Therefore, the detailed proof can be found in \cref{proof}.
		
		\begin{remark}
		 \Cref{thm_mtsm_smooth} can be viewed as generalizing a result of Bergmann and Zimmermann \cite[Section 3.2]{Zimmermann2024multihermi} to the setting of \emph{quasi-interpolation} in the tangent space.
		 By allowing a slight error between $\widehat{f}_j$ and $f$, establishing that the Riemannian Hessians are invertible becomes significantly more complicated. Fortunately, Karcher's analysis of the local convexity of the distance function in \cite{karcher} applies.
		\end{remark}

		\subsection{Error bound}\label{sec_sub_error}
		An error bound of MTSM is obtained by combining \cref{thm_rtsm_unique} with \cref{thm_error_TSM,col:error_TSM} as follows.
		
		\begin{proposition}[Error bound]\label{prop:error_RTSM}
		Under the assumptions of \cref{thm_rtsm_unique},
		\[
		 \dist_{\Mm{}} \bigl( f(x), (\mathscr{M}f)(x) \bigr) \le 5 \epsilon.
		\]
		\end{proposition}
		\begin{proof}
		Let $x\in\Omega$ be arbitrary. As in the proof of \cref{thm_rtsm_unique}, we can assume without loss of generality that $x \in \Omega_1 \cap\dots\cap \Omega_\ell$ for some $1 \le \ell \le R$.
		Then, the proof of \cref{thm_rtsm_unique} revealed that 
		\[      
		 \dist_{\mathcal{M}} \bigl(f(x), \widehat{f}_j(x) \bigr) \le 5\epsilon < \frac{1}{2} \min\left\{\mathrm{inj}(\mathcal{M}), \frac{\pi}{2\sqrt{K}} \right\}, \quad j=1,\dots,\ell.
		\]
		This implies that the geodesic ball $\mathcal{B}_{5\epsilon}(f(x))$ is \emph{convex} by \cite[Theorem 6.4.8]{Petersen}. The main result of Karcher \cite[Theorem 1.2]{karcher} implies that the minimum of \cref{eqn_mtsm} is attained in the interior of this convex ball.
		As $(\mathscr{M}f)(x)$ is the solution of \cref{eqn_mtsm} and the only nonzero $\varphi_j(\widehat{f}_j(x))$'s are those with $x \in \Omega_j$, this concludes the proof.
		\end{proof}
		
		\Cref{prop:error_RTSM} implies that the global approximation error of MTSM depends on the local approximation errors of the individual STSMs. This result is not surprising, as we are essentially employing a partition of unity in \cref{def_mtsm} to smoothly determine which individual STSMs are active while the Fr\'echet mean is used to combine their individual predictions.

		\subsection{Number of tangent spaces}
		The number of tangent spaces can affect the approximation accuracy, as we except that by using more tangent spaces we can reduce their radii $\sigma_j$ which may improve the local approximation error $\epsilon$ in \cref{prop:error_RTSM}.
% 		In this subsection, we consider the number of tangent spaces, which is a crucial factor that we should take into account.
		We first consider a theoretical upper bound on the number of anchor points so that their geodesic balls $\mathcal{B}_{\sigma_j}({p_j^*})$ can cover the image of $f$. This bound can be obtained from a consequence of the \emph{Bishop--Gromov's relative volume comparison theorem} \cite{richard1964bishop}, called \emph{Gromov's Packing Lemma}. The latter provides a bound on the number of small balls to cover a larger ball.
	
		\begin{lemma}[Gromov's Packing Lemma {\cite{gromov1981structures}}] \label{lem_packing}
			Let $\Mm{}$ be an $n$-dimensional Riemannian manifold whose sectional curvatures are lower bounded by $L \in \Rr$. Let $M_L$ be the complete $n$-dimensional simply connected space of constant sectional curvature $L$ and let $q\in M_L$ be any point.
			Given $\Sigma \ge \sigma > 0$, for all $p \in \Mm{}$, there exists a set $\{\mathcal{B}_{\sigma}(p_i) \mid i = 1, \dots, \ell \} \subset \Mm$ with $p_i \in \mathcal{B}_\Sigma(p)$ of at most
			\[
			\mathcal{V}_{\Mm}\left( 2\Sigma, \frac{1}{4}\sigma \right) := \frac{\mathrm{Vol}(\mathcal{B}_{2\Sigma}(q))}{\mathrm{Vol}(\mathcal{B}_{\sigma/4}(q))}
			\]
			balls covering $\mathcal{B}_{\Sigma}(p)$, where $\mathrm{Vol}$ denotes the volume of a geodesic ball on $M_L$.
		\end{lemma}
	
		Incorporating Gromov's Packing Lemma and the error bound of MTSM, the next relationship between the upper bound on the number of tangent spaces and the error bound in MTSM can be observed.
		
		\begin{corollary}\label{up_numberT}
		If we restrict the radii of the smaller geodesic balls to $\sigma_1=\cdots=\sigma_R = \min\{\mathrm{inj}(\Mm), \pi / \sqrt{|L|} \}$ in \cref{thm_rtsm_unique}, then a smooth map $f$ can be approximated with $R$-MTSM under the assumptions of \cref{thm_rtsm_unique} for some value of $R$ satisfying
		\[
		\frac{\mathrm{diam}(f(\Omega))}{2 \sigma_1} \le R \le \mathcal{V}_{\Mm}\left( \mathrm{diam}(f(\Omega)), \frac{1}{4} \sigma_1 \right),
		\]
		where $\mathrm{diam}(S) := \sup_{p,q \in S} \operatorname{dist}_\Mm(p,q)$ for any subset $S \subset \Mm$.
		\end{corollary}
		\begin{proof}   
		The upper bound is immediate from \cref{lem_packing}.
		
		As $f$ is smooth and $\Omega \subset \Rr^n$ is open and connected, the image of a piecewise smooth path in $\Omega$ will be a piecewise smooth path in $f(\Omega) \subset \Mm$. By requirement A1 of \cref{thm_rtsm_unique}, we necessarily need to cover all paths in $f(\Omega)$ using the geodesic balls $\mathcal{B}_{\sigma_j}(p_j^*)$. Given balls of diameter $2 \sigma_1$, this can be accomplished by using a number of balls that is at least the length of the path in $f(\Omega)$ divided by $2\sigma_1$. The maximum distance between points in $f(\Omega) \subset \Mm$ is $\mathrm{diam}(f(\Omega))$ which lower bounds the length of any piecewise smooth path \cite{lee2018rie}. This concludes the proof.
		\end{proof}

		\section{Algorithms}\label{sec:alg}
		In this section, we introduce algorithms to set up and evaluate MTSM.
		The algorithms for setting up MTSM comprise the \emph{offline stage}, which is executed only once for a given approximation problem. The evaluation of MTSM comprises the \emph{online stage}, which is assumed to be executed a large number of times.
		These two stages are detailed respectively in \cref{subsec:offline} and~\cref{subsec:online}.
			
			\subsection{Offline stage}\label{subsec:offline} 
			The offline stage consists of two main steps: (i) compute suitable anchor points $p^*_j$, and (ii) compute the vector-valued functions $\widehat{g}_j$.
			
			\subsubsection*{Step 1}	
			For well-posedness, \cref{thm_rtsm_unique} prescribes that the $R$ anchor points $p_j^* \in \Mm$ satisfy assumption A1: the geodesic balls $\mathcal{B}_{\sigma_j}(p_j^*)$ with radii
			\begin{equation}\label{eqn_radii_limit}
			\sigma_j < \min\{ \mathrm{inj}(p_j^*), \pi / \sqrt{|L|} \}
			\end{equation}
			should cover the image of $f$. In our setting, $f$ is given through samples $S$ as in \cref{eqn_samples}. Verifying whether the balls $\mathcal{B}_{\sigma_j}(p_j^*)$ cover $f(\Omega)$ will thus not be possible. At best we can cover the set $\{ y_1, \dots, y_N \}$. Hence, it seems natural to choose the anchor points that minimize the maximum distance of the $y_i$'s to the nearest anchor point $p_j^*$; i.e.,
			\begin{equation} \label{eqn_max_cluster_dist}
				d^*_R = \min_{p_1^*,\dots,p_R^* \in \Mm} \, \max_{i=1,\dots,N} \, \min_{j=1,\dots,R} \|\mathrm{Log}_{p^*_j} y_i\|_{p^*_j},
			\end{equation}
			where $\|\mathrm{Log}_{p^*_j} y_i\|_{p^*_j} := \infty$ if $y_i$ does not lie in the injectivity radius of $p_j^*$.
			This is a standard clustering problem on a Riemannian manifold. 
			Hence, we propose to choose the anchor points as the cluster centers of the given points $y_i$'s that (approximately) solve \cref{eqn_max_cluster_dist}. For this we can use any existing clustering method, which assigns each point of $Y$ to one (i.e., partitioning) or more (i.e., soft clustering) clusters $C(p_i^*)$.
			
			The foregoing discussion shows that after fixing the number of anchor points $R$, it is sensible to choose them as cluster centers. But how do we determine $R$? We say that a clustering is \emph{covering} if $d_R^* < \infty$ is a finite value in \cref{eqn_max_cluster_dist}, which entails that each data point $y_i$ lies in the injectivity radius of at least one anchor point. To satisfy assumption A1 of \cref{thm_rtsm_unique}, $R$ should be such that the clustering is covering. Moreover, $R$ should also be sufficiently large so that \cref{eqn_radii_limit} holds; that is, $R$ should be any value such that $d_R^* < \pi / \sqrt{|L|}$. If all the other assumptions of \cref{thm_rtsm_unique} also hold, then \cref{prop:error_RTSM} guarantees that the error is bounded by the same quantity $5\epsilon$ regardless of $R$. Thus, to attain any fixed error bound, we propose to choose the smallest $R$ so that $d_R^* < \pi / \sqrt{|L|}$. \Cref{up_numberT} then gives some guidance on the range in which such an $R$ would lie under simplifying assumptions.
		
			\begin{algorithm}[tb]
				\caption{Adaptive anchor points selection}\label{alg:aco2}
				\begin{algorithmic}[1]
					\STATE{\textbf{Input}: Data points $Y = \{y_1,\ldots y_N\} \subset \Mm{}$, the maximal number of anchor points $R_{\max} < N$, an estimate of the injectivity radius $0 \le M \approx \mathrm{inj}(\Mm)$, a negative lower bound $L \le 0$ on the sectional curvatures of $\Mm$.}
					\STATE{\textbf{Output}: The clusters $C(p_i^*) \subset Y$ centered at the anchor points $p^*_i \in \Mm{}$.}
					\STATE{$R_{\min} \leftarrow \left\lfloor \frac{1}{2} \mathrm{ddiam}(Y) \max\{ M^{-1}, \sqrt{|L|}/\pi \} \right\rfloor$}
					\IF{$R_{\min} > R_{\max}$}
					\STATE $R_{\min} \leftarrow 1$
					\ENDIF 
					\FOR{$R = R_{\min},\ R_{\min}+1, \ldots, R_{\max}$}
					\STATE{Compute $R$ clusters $C(p_1^*) \cup \dots \cup C(p_R^*) = Y$ using any (soft or hard) clustering method.}
					\IF{$d_R^* < \pi / \sqrt{|L|}$}\label{alg:crite_anchor}
					\RETURN {$(C(p_1^*), \dots, C(p_R^*)$}
					\ENDIF
					\ENDFOR
					\RETURN {failure}
				\end{algorithmic}  
			\end{algorithm} 
			
			The above observations are summarized as \cref{alg:aco2}. To determine a reasonable initial number of clusters, we approximate the lower bound in \cref{up_numberT}. Neither under- or overshooting the left-hand side of \cref{up_numberT} impacts the correctness; if we undershoot it a few extra cluster sizes will be tested, while if we overshoot it a smaller number of anchor points might have sufficed.
			We use the quantity 
			\begin{equation*}
			\mathrm{ddiam}(Y) := \max_{\substack{x,y \in Y,\\ 0<\dist_\Mm(x,y)<\infty} } \dist_\Mm(x, y)
			\end{equation*}
			as a discrete approximation of the diameter of $f(\Omega)$, and a supplied $M$ as an approximation of the injectivity radius of $\Mm$. If $M = 0$, then $R_{\min}=1$ will be the initial number of clusters.	
			
			\subsubsection*{Step 2}
			Based on the first stage, we approximate the vector-valued function $g_j(x)$ in \cref{alg:aco3}. To compute the radius of the ball $B_{\sigma_j}(p^*_j)$, we suggest to estimate it as in line \ref{alg:line_radius} of \cref{alg:aco3}. 
			
			\begin{algorithm}[tb]
				\caption{Compute vector-valued function}\label{alg:aco3}  
				\begin{algorithmic}[1]
					\STATE \textbf{Input}: Anchor points $p^*_j \in \Mm{}$ and their clusters $C(p_j^*)$ for $j =1,\ldots,R$; and a sample set $S=\{ (x_i, y_i)\in\Omega\times\Mm \mid i=1,\dots,N \}.$
					\STATE{\textbf{Output}: Vector-valued functions $\widehat{g}_j : \Omega_j \to T_{p_j^*}\Mm,$ for $j=1,\dots,R$.}
					\FOR{$j = 1:R$}
					\STATE {Set the radius $\sigma_j := \max_{y \in C(p_j^*)} \|\mathrm{Log}_{p^*_j} y\|_{p^*_j}$} \label{alg:line_radius}
					\STATE{Compute $\widehat{g}_{j}$ from $\bigl\{(x_i,\mathrm{Log}_{p^*_j}y_i)\bigr\}_{i=1}^N$ using vector-valued function approximation methods}
					\ENDFOR
					\RETURN $(\widehat{g}_1, \dots, \widehat{g}_R)$
				\end{algorithmic}  
			\end{algorithm} 
			
			\subsection{Online stage}\label{subsec:online}
			
			The online stage consists of evaluating \cref{eqn_mtsm} for a give $x \in \Omega$. This is given by~\cref{alg:aco4}. The core step of the online stage is computing the weighted Fr\'echet mean, which can be computed, e.g., via Riemannian gradient descent \cite{mnoptbo} or approximated using successive geodesic interpolation \cite{chakraborty2019statistics, Chakraborty15Recursive}. In \cref{sec:Numer_exp}, we use a Riemannian gradient descent method.
			
			\begin{algorithm}[tb]
				\caption{Computing Riemannian multiple centroids (Online)}\label{alg:aco4}  
				\begin{algorithmic}[1]
					\STATE{\textbf{Input}: Anchor points $p^*_j \in \Mm{}$ for $j =1,\ldots,R$; the vector-valued functions $\widehat{g}_{j}$}; and the point $x \in \Omega$.
					\STATE{\textbf{Output}: The approximate function value $(\mathscr{M} f)(x)$.}
					\STATE {Compute weight $\varphi_j(x)$ as \eqref{eqn_specific_weight_function}.}
					\STATE  Define the active site $\mathrm{I}(x) = \{ j \mid \varphi_{j}(x) >0\}$.
					\IF{$\mathrm{I}(x) = \emptyset$}
					    \RETURN failure
					\ENDIF
		            \STATE {Compute $q_j = \mathrm{Exp}_{p_j^*}\widehat{g}_{j}(x)$ for all $j \in \mathrm{I}(x)$ }
					\IF{$\mathrm{I}(x) = \{j\}$}
						\RETURN $q_{j}$
					\ELSE
						\RETURN $ q^* =  \argmin_{p \in \Mm{}} \sum_{j \in \mathrm{I}} \varphi_j(x) \dist_{\Mm{}}(q_j,p)^2$
					\ENDIF
				\end{algorithmic}  
			\end{algorithm} 	
	
	\section{Numerical experiments}\label{sec:Numer_exp}
	In this section, we present numerical experiments to illustrate the performance of the proposed MTSM from \cref{sec:R-TSM} and the algorithms from \cref{sec:alg} as compared to other methods from the literature. We compare MTSM with three manifold-valued approximation schemes: STSM (see \cref{sec_1tsm}), RMLS (see \cref{sec_rmls}), and MRMLS from \cite[Section 5.4]{sharon2023mrmls}. We implemented RMLS based on the description in \cite[Section 4.3]{Grohs17} and MRMLS as described in \cite[Algorithm 5.1]{sharon2023mrmls}. STSM uses the same code as MTSM.
 	
	The general experimental setup is as follows. First, we construct a \emph{training set} $S=\{ (x_i, f(x_i)) \in \Omega\times\Mm \}^N_{i=1}$ for approximating a given $f : \Omega \to \Mm$ as in \cref{eqn_f}. This set is used by STSM and MTSM to build their model offline, while RMLS and MRMLS use it as the set of points whose weighted Fr\'echet mean is taken. Second, an independent \emph{test set} is generated to assess the performance of the constructed models online. In each experiment, all algorithms use the same training and test sets.

	\subsection{Implementation details}
	All algorithms were implemented as Matlab code. Our implementation including the numerical experiments is available at \url{https://gitlab.kuleuven.be/numa/public/mtsm}.
	All the experiments were performed on an HP EliteBook x360 830 G7 Notebook with 32 GiB memory and an Intel Core i7-10610U CPU using Matlab 2022a on Ubuntu 22.04.3 LTS. Matlab was allowed to use all four physical CPU cores.

	The setup of MTSM (and STSM as a special case) used throughout the experiments is as follows. We perform the offline stage of MTSM with \cref{alg:aco3} using a straightforward \emph{Riemannian $k$-means clustering algorithm} as hard clustering method; we transformed the implementation from Geomstat \cite{miolane2020Geomstats} into Matlab code. The online stage of MTSM is applied to the test set with \cref{alg:aco4}.
	Let $d_j^2(x) = \dist_\Mm{}(p^*_j, \widehat{f}_{j}(x))^2$ be the squared distance to the $j$th anchor point. This is a smooth function from $\Omega \to \Rr$ if $\widehat{f}_j : \Omega \to \Mm$ is smooth. Then, the smooth weight function we choose for MTSM is
	\begin{equation}\label{eqn_specific_weight_function}
	\varphi_j(x)  = \frac{h_j(d_j^2(x))}{\sum_{j=1}^{R}h_j(d_j^2(x))},\quad\text{where}\quad
	h_j(d) = \frac{e^{-(\sigma_j^2-d)_+^{-1}}}{e^{-(\sigma_j^2-d)_+^{-1}} + e^{-(d-c\sigma_j^2)_+^{-1}}}
	\end{equation}
	is a smooth cutoff function \cite[Chapter 2]{lee2003smooth}, $\sigma_j < \mathrm{inj}(p_j^*)$, $(x)_+ := \max\{0, x\}$ for $x\in\Rr$, $0 < c < 1$, and $e^{-1/0} = 0$. A cutoff function $h_j$ is visualized in \cref{fig_cutoff}. The value $c = \tfrac{1}{4}$ is used in the experiments.

	\begin{figure}[tb]
	\includegraphics{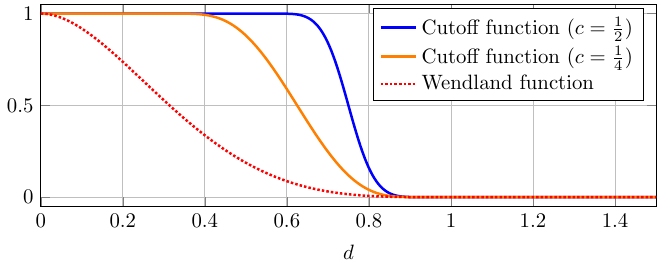}
	\caption{The Wendland function $\varphi(d)$ and the smooth cutoff function $\varphi_j(d)$ with $\sigma_j^2=1$ and $c=\frac{1}{2}$.}
	\label{fig_cutoff}
	\end{figure}

	For MRMLS and RMLS, the Wendland function \cite{PLI95} $\varphi(d):=(1-d)_{+}^4(4d+1)$ is used as weight function. It is visualized in \cref{fig_cutoff}. 
% 	We will use $A$ as a matrix and  $\mathbf{a}$ as a vector in the following sections.

	The weighted Fr\'echet mean in RMLS, MRMLS and MTSM is computed by a trust region based optimization method \cite{mnoptbo} with the \texttt{centroid} function from the Manopt toolbox \cite{manopt} in Matlab.

	\subsection{Basic examples}\label{subsec:toy_example}

	We first compare MTSM to several manifold-valued function approximation methods from the literature on two simple benchmark problems involving matrix manifolds $\Mm$. In both cases, we measure the accuracy of the approximant $\widehat{f}:\Omega\to\Mm$ to $f:\Omega\to\Mm$ via the maximum relative error%
	\begin{equation}\label{eqn_relerr_matrix}
	 \mathrm{relErr}(f, \widehat{f}) := \max_{(\mathbf{x}, f(\mathbf{x})) \in T} \frac{\dist_{\Mm}(f(\mathbf{x}), \widehat{f}(\mathbf{x}))}{\|f(\mathbf{x})\|_F}
	\end{equation}
	on the test set $T \subset \Omega\times\Mm$, where $\|\cdot\|_F$ denotes the Frobenius norm of a matrix.

	\subsubsection{Symmetric positive-definite matrices}\label{subsec:spdfun}
	We consider the function that maps $[-1, 1]^2$ into the manifold of $3\times3$ symmetric positive-definite matrices $\mathcal{P}_3 \subset \Rr^{3\times 3}$ from \cite[Section 6.5.2]{sharon2023mrmls}:
	\begin{align}
		\nonumber f : [-1,1]^2 &\longrightarrow \mathcal{P}_3,\\
		\label{spd_fun} \mathbf{x} &\longmapsto 
		2 I  + (|\cos(2 x_2)+0.6|) e^{- x_1^2 - x_2^2}
		\begin{pmatrix}
			10+2\sin (5 x_2 ) & x_2 	& x_1 x_2 \\
			x_2 			  & 10 		& x_2^2 \\
			x_1 x_2			  & x_2^2 	& 10
		\end{pmatrix}.
	\end{align}
	We equip $\mathcal{P}_3$ with the affine invariant metric that is discussed in \cite[Chapter 11, section 7, equation (11.35)]{mnoptbo}. Then, the exponential and logarithmic maps are
	\begin{equation*}
	\begin{aligned}
		\operatorname{Exp}_P(\dot{P}) &= P^{1 / 2} \exp_m \left(P^{-1 / 2} \dot{P} P^{-1 / 2}\right) P^{1 / 2},\\
		\operatorname{Log}_P(Q) &= P^{1 / 2} \log_m \left(P^{-1 / 2} Q P^{-1 / 2}\right) P^{1 / 2},
	\end{aligned}
	\end{equation*}
	where $P, \ Q \in \mathcal{P}_3$; $\dot{P} \in T_P{\mathcal{P}_3}$; $P^{1/2}$ is the Cholesky factor of $P$; $\mathrm{exp}_m(\cdot)$ is the classic matrix exponential; and $\mathrm{log}_m(\cdot)$ is the classic matrix logarithm; see \cite[equations (11.37) and (11.38)]{mnoptbo}.

	We investigate the impact on the accuracy of STSM, MTSM, RMLS, and MRMLS with a varying number of training samples. For this, we generate training sets
	\[
		S_k = \{ (\mathbf{x}_i, f(\mathbf{x}_i)) \in [-1,1]^2 \times \mathcal{P}_3 \mid i=1,2,\dots, \lfloor 50 \cdot 1.5^k \rfloor \}
	\]
	for $k=0,1,\dots,5$, where the $\mathbf{x}_i$'s are chosen as the first points from the Halton sequence as generated by Matlab's \texttt{haltonset(2)} function. As test set we choose a $50 \times 50$ uniform grid on $[-1,1]^2$.

 	For MTSM (including STSM), we choose radial basis functions interpolation from \cite{Beckert2001MultivariateIF,biancolini2017fast} based on the multiquadric radial basis function in \cite[Chapter 2.2]{biancolini2017fast} to approximate the vector-valued function in the offline stage. For this, we used the implementation from \url{https://github.com/joslorgom/RBFinterp}. For all training sets, we use $L = -4$ as lower bound of the sectional curvature of $\mathcal{P}_3$~\cite[Proposition I.1]{cris2020sectioncurvature}; with this choice the stopping criterion in line \ref{alg:crite_anchor} of~\cref{alg:aco2} terminated at $R=3$ anchor points.

	\begin{figure}[tbp]
		\footnotesize
		\centering
		\includegraphics[width=1\linewidth]{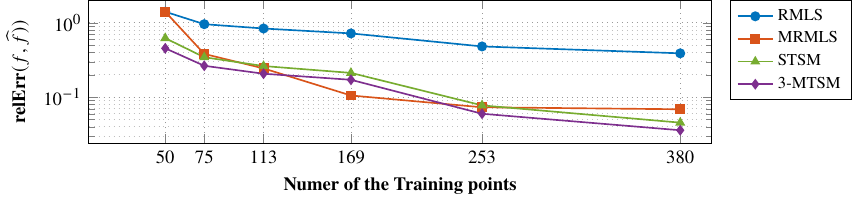}
		\caption{The maximum relative error \cref{eqn_relerr_matrix} of the function~\cref{spd_fun} approximated by RMLS, MRMLS, STSM and $3$-MTSM. The vertical axis is displayed in a logarithmic scale. The setup is described in \cref{subsec:spdfun}.}\label{fig:SPD}
	\end{figure}

	\Cref{fig:SPD} shows the relative error of the function~\cref{spd_fun} by the different methods for increasingly larger training sets. We observe the relative approximation error of MTSM is approximately one order of magnitude lower than that of RMLS, and similar to those of MRMLS and STSM.

	The computation time of the offline stage in MTSM is approximately $0.20$ seconds, while STSM requires about $0.022$ seconds. For the online stage based on the largest training set $S_5$, predicting the $2500$ test points takes less than $0.5$ seconds for both RMLS and STSM. By constrast, MTSM takes around $4$ times as long ($2$ seconds) and MRMLS approximately $9$ times as long ($\approx 4.5$ seconds). 

	\subsubsection{Special orthogonal group}\label{subsubsec:SO3} Using MTSM, multivariate Hermite interpolation in a single tangent space (THI) from \cite[Section~4]{Zimmermann2024multihermi}, and barycentric Hermite interpolation (BHI) from \cite[Section~3]{Zimmermann2024multihermi} we will approximate the function from \cite[Section~5.2]{Zimmermann2024multihermi}:
	\begin{equation}\label{eq:rota}
			f: \Rr^2 \longrightarrow \mathcal{SO}(3),\quad  \left(x_1, x_2\right) \longmapsto \exp _m H\left({x}_1, {x}_2\right),
	\end{equation}
	where $\mathcal{SO}(3) = \{ Q \in \Rr^{3\times3} \mid Q^T Q = Q Q^T = I, \mathrm{det}(Q) = 1 \}$ is the \emph{special orthogonal group}, which we view as an embedded Riemannian submanifold of the Euclidean space $\Rr^{3\times 3}$; $\cdot^T$ denotes the matrix transpose; $\mathrm{exp}_m$ is the classic matrix exponential; and
	\begin{equation*}
		H(x_1, x_2)= \begin{pmatrix}
			0 & {x}_1^2+\frac{1}{2}{x}_2 & \sin \left(4 \pi\left({x}_1^2+{x}_2^2\right)\right) \\[2pt]
			-{x}_1^2-\frac{1}{2} {x}_2 & 0 & {x}_1+ {x}_2^2 \\[2pt]
			-\sin \left(4 \pi\left({x}_1^2+{x}_2^2\right)\right) & -{x}_1- {x}_2^2 & 0
		\end{pmatrix}.
	\end{equation*}
	In this geometry of $\mathcal{SO}(3)$, the exponential and logarithmic maps are, respectively,
	 	\begin{equation*}
	 			\operatorname{Exp}_P(\dot{P})=P \exp_m (P^{T} \dot{P})
	 			\quad\text{and}\quad
	 			\operatorname{Log}_P(Q)=P\log_m ( P^{T}Q ),
	\end{equation*}
	where $P, Q \in \mathcal{SO}(3)$; $\dot{P} \in T_P{\mathcal{SO}(3)}$; and $\mathrm{log}_m(\cdot)$ is the classic matrix logarithm.

    We will investigate the performance of the methods on two training sets. We use a $7 \times 7$ Chebychev grid on the domain $[-0.5,0.5]^2$ and its corresponding function values (represented as $3\times 3$ matrices) as first training set $S_1$, and a $14 \times 14$ Chebychev grid on $[-1,1]^2$ and its corresponding function values as a second training set $S_2$. As test sets, we respectively use a $20 \times 20$ uniform grid on $[-0.5,0.5]^2$, and a $20 \times 20$ uniform grid on $[-1,1]^2$.

    The implementation details for BHI and THI are outlined in \cite[Section~5.2]{Zimmermann2024multihermi}. We used their Matlab implementation from \url{https://github.com/RalfZimmermannSDU/MultivarHermiteManifoldInterp_SISC}. As for THI, we randomly sampled one of the training points on the manifold as the location of the tangent space.
    For MTSM, we approximate the vector-valued function $\widehat{g}_j(x)$ using multivariate Hermite interpolation, as detailed in \cite[Section 4]{Zimmermann2024multihermi}. In this configuration, MTSM can be viewed as an extension of THI, but now using multiple tangent spaces. Running \cref{alg:aco2}, it decides to use $R=2$ tangent spaces for both training sets when $L=-1$ is supplied as lower curvature bound.
    
	\begin{figure}[tbp]
		\footnotesize
		\centering
		\begin{subfigure}[b]{1\textwidth}
			\includegraphics[width=1\linewidth]{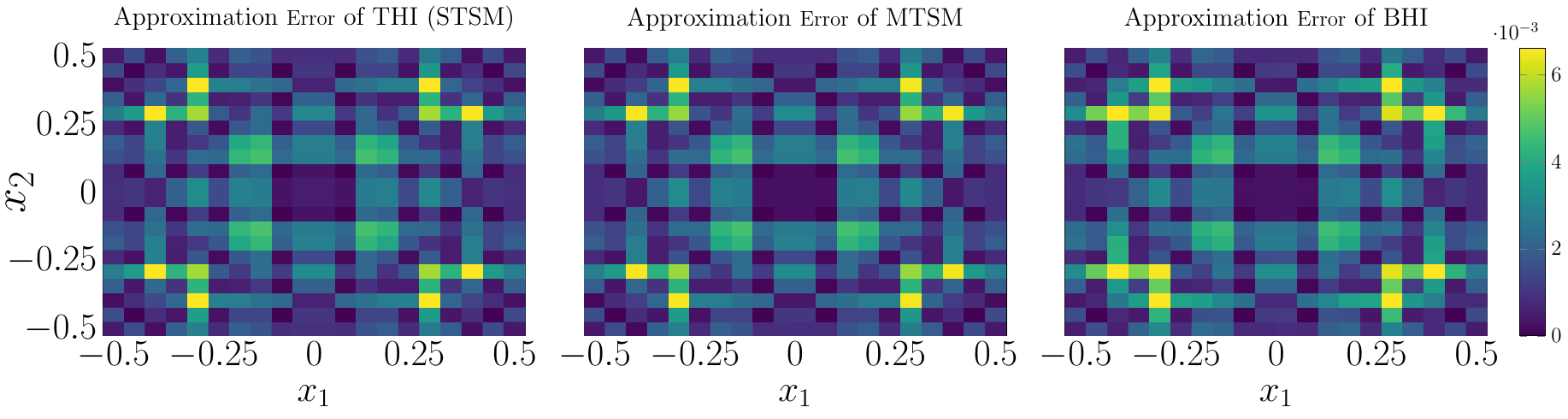}
			\caption{The $20\times20$ uniform grid on $[-0.5,0.5]^2$.}
		\end{subfigure}
		\\
		\begin{subfigure}[b]{1\textwidth}
			\includegraphics[width=1\linewidth]{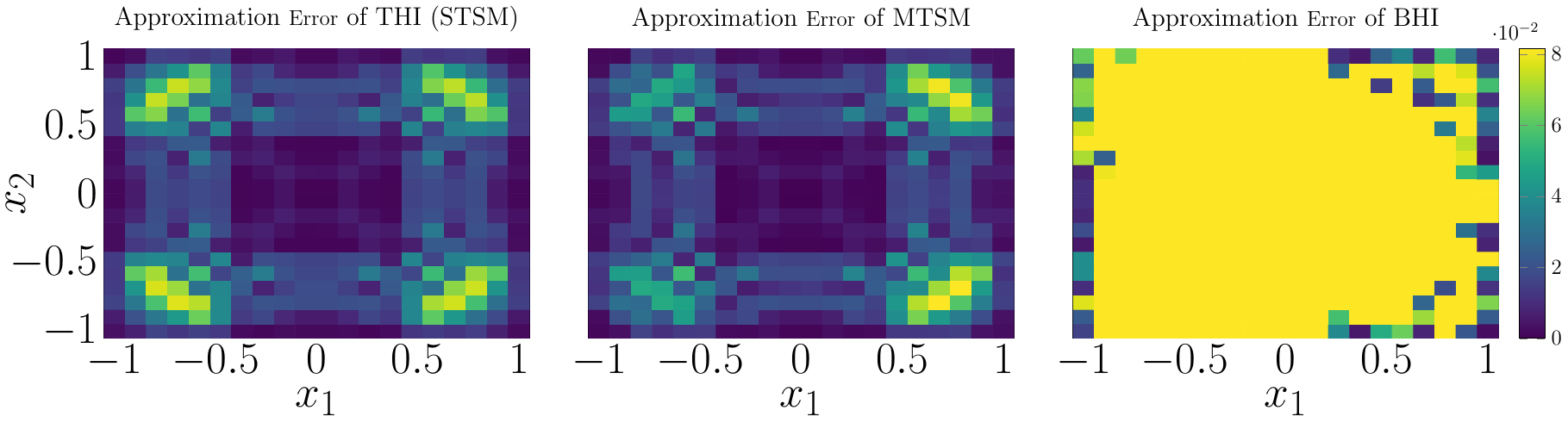}
			\caption{The $20\times20$ uniform grid on $[-1,1]^2$.}
		\end{subfigure}
		\caption{The relative error \cref{eqn_relerr_matrix} for the $\mathcal{SO}(3)$-valued function \cref{eq:rota}, approximated in different domains by different models. The scale used is the same for all panels in subfigure (a), and likewise for subfigure (b). The setup is described in \cref{subsubsec:SO3}.}\label{fig:SO3-valued}
	\end{figure}

	\Cref{fig:SO3-valued} shows the relative error \cref{eqn_relerr_matrix} for approximating $f$ in \cref{eq:rota} of $3$-MTSM, BHI and THI for approximating the $\mathcal{SO}(3)$-valued function~\cref{eq:rota}. When the function~\cref{eq:rota} is defined on $[-0.5,0.5]^2$, the approximation errors of $2$-MTSM, THI, and BHI are very similar. However, on $[-1,1]^2$, BHI has a significantly higher approximation error compared to both THI and $2$-MTSM, whose results are rather similar.
% 	And on domain $[-1,1]^2$, $3$-MTSM have better approximation than THI.

	The computation time for the offline stage using the $2$-MTSM is approximately $487$ seconds, whereas THI requires only about $197$ seconds to build the model. In the online stage, predicting the $400$ test points on the $[-1,1]^2$ domain takes about $38$ seconds for THI, $90$ seconds for $2$-MTSM, and $2$ hours for BHI.
	\subsection{Parametric model order reduction}\label{subsec:pMOR}

	The goal of parametric model order reduction (pMOR) is to replace a complicated system by a simpler one that can be computed efficiently. This is often achieved by first sampling a set of parameters, then computing reduced order models at these points, and finally using some form of (quasi-)interpolation to extend the approximation to the entire parameter domain.
	
	We consider a linear time-varying dynamical system parameterized with parameter $\mathbf{p} \in \Omega \subset \Rr^d$ (usually, $\Omega$ is a bounded domain):
	\begin{equation}\label{Eq:1}
		\begin{aligned} 
			{E}(\mathbf{p}) \dot{\mathbf{x}}(t; \mathbf{p}) &= {A}(\mathbf{p}) \mathbf{x}(t ; \mathbf{p}) + {B}(\mathbf{p}) \mathbf{u}(t), \quad \text { with } \quad \mathbf{x}(0; \mathbf{p})=\mathbf{0}, \\
			\mathbf{y}(t ; \mathbf{p}) &= {C}(\mathbf{p}) \mathbf{x}(t ; \mathbf{p}),
		\end{aligned}
	\end{equation}
	where $t \in [0, \infty)$ and we used the dot-notation $\dot{\mathbf{x}}(t; \mathbf{p})$ to denote the time-derivative $\frac{\mathrm{d}}{\mathrm{d}t} \mathbf{x}(t; \mathbf{p})$. The state variables $\mathbf{x}(t ; \mathbf{p}) \in \Rr^n$, are determined by the given inputs $\mathbf{u}(t) \in \Rr^m$, the matrices ${A}(\mathbf{p}), E(\mathbf{p}) \in \Rr^{n \times n}$, ${B}(\mathbf{p}) \in \Rr^{n \times m}$, and ${C}(\mathbf{p}) \in \Rr^{q \times n}$. The output of the system is $\mathbf{y}(t ; \mathbf{p}) \in \Rr^q$.
	
	The idea of pMOR is to approximate the original state variables inside of a low-dimensional vector subspace, i.e., $\mathbf{x}(t;\mathbf{p}) \approx {V}(\mathbf{p})\mathbf{x}_r(t;\mathbf{p})$, where $\mathbf{x}_r(t; \mathbf{p}) \in \Rr^r$ and ${V}(\mathbf{p}) \in \Rr^{n \times r }$ has orthogonal columns. Then, multiplying~\cref{Eq:1} from the left by a matrix with orthonormal rows ${W}(\mathbf{p})^T \in \Rr^{n \times r}$, we obtain the reduced system:
	\begin{align*}
		{E}_r(\mathbf{p}) \dot{\mathbf{x}}_r(t ; \mathbf{p}) &= {A}_r(\mathbf{p}) \mathbf{x}_r(t ; \mathbf{p}) + {B}_r(\mathbf{p}) \mathbf{u}(t), \quad \text { with } \quad \mathbf{x}_r(0 ; \mathbf{p})=\mathbf{0}, \\
		\mathbf{y}_r(t ; \mathbf{p}) &= {C}_r(\mathbf{p}) \mathbf{x}_r(t; \mathbf{p}),
	\end{align*}
	where the reduced matrices are
	\begin{align*}
	 {E}_r(\mathbf{p}) &:= {W}(\mathbf{p})^T {E}(\mathbf{p}) V(\mathbf{p}) \in \Rr^{r\times r}, &&&
	 {B}_r(\mathbf{p}) &:= {W}(\mathbf{p})^T {B}(\mathbf{p}) \in \Rr^{r \times m}, \\
	 {A}_r(\mathbf{p}) &:= {W}(\mathbf{p})^T {A}(\mathbf{p}) V(\mathbf{p}) \in \Rr^{r\times r}, &&&
	 {C}_r(\mathbf{p}) &:= {C}(\mathbf{p}) {V}(\mathbf{p}) \in \Rr^{q \times r}.
	\end{align*}
	Without going into further detail, substituting $V(\mathbf{p}) \mapsto V(\mathbf{p})Q(\mathbf{p})$, where $Q(\mathbf{p})$ is a square orthogonal matrix, alters only the representation of the reduced system but not its approximation properties. The pMOR approximation problem is hence intrinsically one on the \emph{Grassmannian}, which is the manifold of all linear subspaces of a fixed dimension, namely
	\[
	\mathcal{G}(n,r) := \{ \mathrm{span}(Q) \mid Q \in \mathrm{St}(n,r) \}, \;\text{ where }\;
	\mathrm{St}(n,r) := \{ Q \in \Rr^{n \times r} \mid Q^T Q = I_r \},
	\]
	rather than on the \emph{Stiefel manifold} $\mathrm{St}(n,r)$ of matrices with orthonormal columns. Elements of the Grassmannian have a natural representation on $\mathrm{St}(n,r)$: the subspace $[Q] := \mathrm{span}(Q) \in \mathcal{G}(n,r)$ can be represented by any orthonormal basis $Q \in \mathrm{St}(n,r)$ of it. Detailed information on the exponential map on the Grassmannian can be found in \cite{Edelman1998Grassexp}, while the logarithmic map can be approximated by the algorithm in \cite[Algorithm 7.10]{zimmermann2021manifold}. The tangent space at a point $P \in \mathcal{G}(n, k)$ is
	\[
	T_P \mathcal{G}(n, k)=\left\{X \in \mathbb{R}^{n \times k} \mid X^{T} P+P^{T} X=0\right\}.
	\]
% 	where  $\cdot^{{H}}$ denotes the complex conjugate transpose.

	We can view pMOR as a problem of approximating the manifold-valued function
	\begin{align*}
			f : \Rr^d \supset \Omega  \longrightarrow \mathcal{G}(n,r) \times \mathcal{G}(m, r), \quad \mathbf{p}  \longmapsto \bigl( [V(\mathbf{p})], [W(\mathbf{p})] \bigr)
	\end{align*}
	such that the solution of the reduced system is a good approximation of the solution of the original system. With an approximation $\widehat{f}$ of $f$, we can efficiently construct and solve the projected small-scale linear dynamical system, yielding the coordinates of the solution $\mathbf{x}(t, \mathbf{p})$ relative to the chosen matrix with orthonormal columns $V(\mathbf{p})$ representing the approximating subspace $[V(\mathbf{p})]$.

	To measure the quality of reduced order models for linear dynamical systems, the \emph{transfer function}
	\[
	{G}(s)
	:= {G}\left(s, \mathbf{p}\right)
	:= {C}\left(\mathbf{p}\right)\left(s {E}\left(\mathbf{p}\right)-{A}\left(\mathbf{p}\right)\right)^{-1} {B}\left(\mathbf{p}\right)
	\]
	is used in the frequency domain \cite{benner2015pmor}. 
	The transfer function of the of reduced system is denoted by ${G}_r(s)$.
	The relative $\mathcal{H}_\infty$-norm \cite{12} is used to measure the accuracy of the reduced-order systems:
	\begin{equation} \label{eqn_relerr_transfer}
	\mathrm{relErr}(G, G_r) := \frac{\left\|{G}-{G}_r\right\|_{\mathcal{H}_\infty}}{\left\|{G}\right\|_{\mathcal{H}_\infty}},
	\quad\text{where}\quad
	\left\| {F} \right\|_{\mathcal{H}_\infty} = \sup_{\omega \in \Rr} \|{F}(\imath \omega) \|_\infty
	\end{equation}
   and $\imath^2 = -1$ is the imaginary unit.\footnote{We compute it using the \texttt{norm(system, `Inf')} function in Matlab's Control System Toolbox.}
	
 	To evaluate the performance of MTSM for solving pMOR problems, we investigate two model problems and compare MTSM's accuracy and computational time with RMLS, STSM, and the interpolatory projection method (INP) from \cite{baur2011pmor}. The latter is implemented in the Matlab psssMOR toolbox \cite{psssmor}. Note that INP was specifically developed to address pMOR problems.

	 The anchor point of STSM was sampled randomly with uniform probability from our training set, as in \cite[Chapter 7]{zimmermann2021manifold}. The vector-valued function $g_j$ in MTSM is approximated by two different schemes: component-wise multivariate linear interpolation, and component-wise multivariate tensorized interpolation, which extends a univariate approximation scheme to the multivariate case through a tensor decomposition, as explained in \cite{simon2024approx,strossner2024approximation}.

	\subsubsection{Anemometer model}\label{subsec:Anemo}
	We consider the anemometer model, a flow sensing device \cite{moosmann2007paramor} modeled by the convection-diffusion partial differential equation
	\begin{equation*}
		\rho c \frac{\partial T}{\partial t}=\nabla \cdot(\kappa \nabla T)-\rho c v \nabla T+\dot{q},
	\end{equation*}
	where $\rho$ is the mass density, $c$ is the specific heat capacity, $\kappa$ is the thermal conductivity, $v$ is the fluid velocity, $T$ is the temperature, and $\dot{q}$ is the heat flow into the system caused by the heater, $\nabla$ denotes the gradient of the $T$ and $\nabla \cdot$ denotes the divergence $\kappa \nabla T$. In 1D space, this equation models how the temperature $T$ changes over time $t$ and space $x$ due to the combined effects of convection (movement with the fluid flow) and diffusion (spreading due to temperature gradients).
	This equation has been discretized with finite differences, yielding a parameterized linear dynamical system in state-space form with a single output \cite{morwiki_anemom}:
	\[
	\begin{aligned}\label{eq:ame}
		{E} \dot{\mathbf{x}}(t) &= \left( {A}_1+ {p}\left({A}_2-{A}_1\right)\right) \mathbf{x}(t) + {B}\mathbf{u}(t) \\
		{y}(t) & = C \mathbf{x}(t)
	\end{aligned}
	\]
	where $\mathbf{p} \in \Rr$ is the parameter, and the matrices ${A}_{1}, A_2, {E} \in \Rr^{n \times n}, B \in \Rr^{n \times 1}, C \in \Rr^{1 \times n}$, where $n=29\,008$.

	As training set, we used $200$ points $x_i$ uniformly spaced in $[0.5, 1.5]$ and used the rational Krylov projection method from \cite{beattie2017mor} to obtain approximating $20$-dimensional vector spaces $[V(\mathbf{p}_i)]$ and $[W(\mathbf{p}_i)]$, both in the Grassmannian $\mathcal{G}(29008, 20)$. As test set, we generate $100$ points uniformly spaced between $0.5$ and $1.5$.
	
	The vector-valued function in MTSM (and STSM) is approximated by multivariate linear interpolation. Supplying $L = -4$ as lower sectional curvature bound, \Cref{alg:aco2} applied to the training set decides to chose $R=10$. 

	\begin{figure}
		\centering
		\includegraphics[width=1\linewidth]{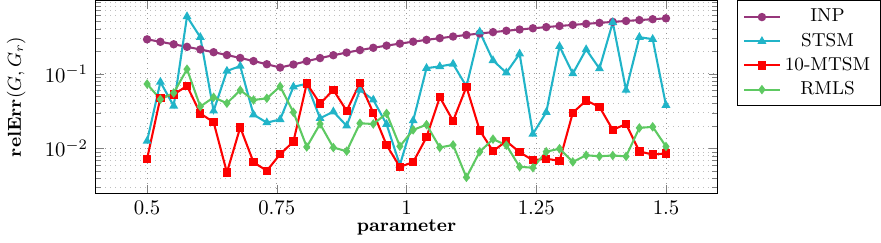}
		\caption{The relative error \cref{eqn_relerr_transfer} between the original anemometer system and its reduced system computed by INP, STSM, $3$-MTSM, and RMLS. The vertical axis uses a logarithmic scale to show the range of relative error values more clearly. The setup is described in \cref{subsec:Anemo}.}\label{fig:anemo1d}
	\end{figure}
	\Cref{fig:anemo1d} displays the relative error \cref{eqn_relerr_transfer} for the different methods.
	It clearly shows that all manifold-valued approximation methods outperform the matrix interpolation method \cite{baur2011pmor}.
	Comparing $10$-MTSM with STSM, we see a substantial advantage in approximation accuracy of MTSM.

	While the accuracy of MTSM is comparable to that of RMLS, as we see in \cref{fig:anemo1d}, $10$-MTSM requires the computation of the weighted Fr\'echet mean of ten points in the online stage, whereas RMLS computes it for all points with nonzero weights, resulting in significantly higher computational costs.
    Specifically, 10-MTSM's offline stage takes approximately $602$ seconds, with around $572$ seconds spent on selecting the $10$ anchor points, and $30$ seconds to approximate the vector-valued functions. By contrast, STSM requires only about $3$ seconds in the offline stage. For the online stage, predicting the bases $[V(\mathbf{p}_i)]$ and $[W(\mathbf{p}_i)]$ at $100$ test points takes about $220$ seconds with $10$-MTSM, about $2.7$ seconds with STSM, $1.1$ seconds with INP, and about $1300$ seconds with RMLS.

	\subsubsection{Microthruster unit model}\label{subsec:pmor3d}

	We consider the 3D microthruster unit model from \cite{PMOR.2}. It describes thermal conduction in a semiconductor chip, offering flexibility in the boundary conditions to simulate environmental temperature changes.
	Discretization leads to a system of linear ordinary differential equations:
	\begin{equation*}
		\begin{split}
		{E} \dot{\mathbf{x}}(t) &= \left({A}-\mathbf{p}(1) {A}_{\text{top}} - \mathbf{p}(2) {A}_{\text {bottom}} - \mathbf{p}(3) {A}_{\text{side}}\right) \mathbf{x}(t)+ \mathbf{b}, \\
		\mathbf{y}(t) &= {C}\mathbf{x}(t),
		\end{split}
	\end{equation*}
	where the heat capacity matrix ${E} \in \Rr^{n \times n}$ and the heat conductivity matrix ${A} \in \Rr^{n \times n}$ are symmetric sparse matrices of order $n=4257$, $\mathbf{b}\in\Rr^{n}$ is the load vector, ${C} \in \Rr^{1 \times n}$ is the output matrix, ${A}_{\text{top}}$, ${A}_{\text{bottom}}$, and $A_{\text{side}}$ are real $n\times n$ diagonal matrices originating from the discretization of the convection boundary conditions.
	
	As training set, we generate $343$ points $\mathbf{p}_i$ on a $7\times7\times 7$ Chebyshev grid on $[10,1000]^3$ and compute the corresponding projection spaces $[{V}(\mathbf{p}_i)]$ and $ [{W}(\mathbf{p}_i)]$ $ \in \mathcal{G}(4257,40)$ as the rational Krylov subspaces computed by the method of \cite{beattie2017mor}. The test sets are the uniform grids with dimensions $k\times k\times k$ for $k=1,2,\dots,8$ on $[10,1000]^3$. Evaluating RMLS took almost $2$ hours on the $8\times8\times8$ uniform test grid, so we did not try $k=9$.
		
	The vector-valued function in MTSM and STSM is approximated by two schemes.
	In STSM(ho) and MTSM(ho), we use the multivariate tensorized method with ST-HOSVD approximation \cite{ST-HOSVD}.
	The ST-HOSVD is implemented with \texttt{hosvd} function in the Tensor Toolbox v3.6 \cite{tensor_toolbox_2023} using the truncation ranks $(90, 30, 7, 7, 7)$ and mode processing order $(1,2,3,4,5)$.
	In STSM(lin) and MTSM(lin) we use multivariate linear interpolation.
	\Cref{alg:aco2} applied to the training set selected $R=2$ anchor points when $L=-1$ is supplied as lower curvature bound. INP is constructed through the psssMOR toolbox using the same training set.
	
	\begin{figure}[t]
		\centering
		\includegraphics[width=1\linewidth]{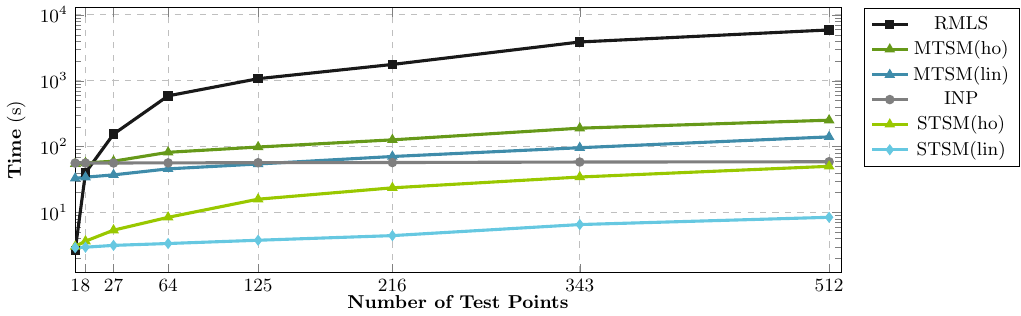}
		\caption{The computational time for approximating $\mathbf{p} \mapsto \bigl( [V(\mathbf{p})], [W(\mathbf{p})] \bigr)$ for solving the Microthruster 3D model based on different number of test points by the methods described in \cref{subsec:pmor3d}. The vertical axis is displayed in a logarithmic scale.}
		\label{fig:timepmor3d}
	\end{figure}

	The computation time for predicting the Krylov subspaces with varying numbers of test points is plotted in \cref{fig:timepmor3d}. The intercept with the vertical axis represents the time taken to perform the offline stage, i.e., to build the models. RMLS has no offline setup cost, while MTSM and STSM need to approximate the vector-valued pullback functions and INP needs to compute the global basis. The time of MTSM is higher than for STSM as we need to compute the anchor points by~\cref{alg:aco3}. These account for the differences in execution time of the offline stage.

	As for the online stage, we see from \cref{fig:timepmor3d} that there is a significant difference between RMLS and all the other methods.
 	RMLS needs to compute the weighted Fr\'echet mean of all points on the Grassmannian with non-zero weights, which results in the highest computational cost. This is mainly attributable to the high cost of computing the gradient of weighted Fr\'echet mean. For $2$-MTSM this cost is significantly reduced because the Fr\'echet mean of at most $2$ points on the Grassmannian will be required. STSM relies on only one evaluation of the exponential map, avoiding the expensive computation of the Fr\'echet mean. This mostly explains the performance difference of about $60\%$ between STSM and $2$-MTSM.

	\Cref{tab:errorpmor3d} presents the maximum and geometric mean value of the relative error \cref{eqn_relerr_transfer} on the largest test set. All methods show comparable performance in terms of accuracy on average. However, we see that the maximum error is significantly larger for INP compared to the manifold-based function approximation methods. All of the MTSM-type methods have a better accuracy than all of the TSMs.

	\begin{table}[tb] \footnotesize
	\centering
	\begin{tabular}{lcc}
		\toprule
		\multicolumn{1}{c}{\textbf{Method}} &\multicolumn{2}{c}{\textbf{Relative error}}  \\
		\cmidrule{2-3}
		&Maximum  & Geometric mean       \\
		\midrule
		INP        &$2.6 \cdot 10^{-1}$    & $7.5  \cdot 10^{-3}$    \\

		RMLS       &$1.6  \cdot 10^{-2}$    & $3.2  \cdot 10^{-3}$     \\
		STSM(lin) &$5.4 \cdot 10^{-2}$   &$3.7 \cdot 10^{-3}$   \\
% 		STSM(cp) &$1.0  \cdot 10^{-1}$   & $5.1  \cdot 10^{-3}$     \\
 		STSM(ho)  &$1.0  \cdot 10^{-1}$   &$5.3  \cdot 10^{-3}$       \\
		$2$-MTSM(lin) &$\mathbf{1.2  \cdot 10^{-2}}$    & $\mathbf{3.1 \cdot 10^{-3}}$     \\
%      	$2$-MTSM(cp) &$1.0  \cdot 10^{-1}$   & $4.3 \cdot 10^{-3}$   \\
     	$2$-MTSM(ho)  &$4.5 \cdot 10^{-2}$  & $3.4 \cdot 10^{-3}$      \\
		\bottomrule
	\end{tabular}
	\caption{The geometric mean and maximum of the relative error \cref{eqn_relerr_transfer} for computing $512$ test points as described in \cref{subsec:pmor3d}.}\label{tab:errorpmor3d}
	\end{table}

	Taking both accuracy and computational efficiency into account, we conclude that MTSM is competitive with the speed of INP and STSM while providing the approximation accuracy of RMLS. This makes it an effective approach for generating qualitative projection spaces for pMOR.

	\section{Conclusions}\label{sec:conc}
	In this paper, we proposed MTSM in \cref{def_mtsm}, a novel function approximation scheme that approximates a manifold-valued function $f$ as in \cref{eqn_f} from samples \cref{eqn_samples} by smoothly mixing predictions of multiple STSMs through the weighted Fr\'echet mean. It combines the computational efficiency of STSM \cite{simon2024approx} and the global approximation ability of the Riemannian moving least squares method \cite{Grohs17}. The numerical experiments show that MTSM can solve different types of problems.
	The \emph{offline} stage in which the model is built has a relatively acceptable cost, provided sufficiently many model queries are made during the online stage. The \emph{online} stage is very cheap compared to RMLS and MRMLS, especially for a large number of points.
	
	Two avenues of further research can be pursued. First, the choice of the location of the anchor points deserves a deeper study. On the one hand, cheaper schemes than Riemannian $k$-means clustering are of interest to accelerate the sometimes costly offline stage. On the other hand, optimizing the location of the anchors through Riemannian optimization could increase the accuracy of MTSM.
	Second, if the samples \cref{eqn_samples} can be chosen as part of the approximation scheme, the (adaptive) choice of these points will affect the approximation accuracy. How the samples should be chosen to satisfy a desired error bound or error decay with MTSM is an open question.

\appendix
\section{Proof of \cref{thm_mtsm_smooth}}\label{proof}
	Our main source for the material to prove \cref{thm_mtsm_smooth} in the remainder of this section is the book by Petersen \cite{Petersen}.
	
		Let $(\mathcal{M},\mathbf{g})$ be a Riemannian manifold. A (smooth) \emph{vector field} $X : \mathcal{M} \to T\mathcal{M}$ on a smooth manifold $\mathcal{M}$ is a smooth section of the \emph{tangent bundle} $T\mathcal{M} = \{ (p,v_p) \mid p\in\mathcal{M}, v_p \in T_p \mathcal{M} \}$, i.e., it is a smooth assignment to $p$ of a tangent vector $X(p) \in T_p \mathcal{M}$. An example of a vector field is the gradient (field) of a smooth function $f : \mathcal{M} \to \mathbb{R}$. Recall that the \emph{gradient} $\nabla f$ at $p\in\mathcal{M}$ is the tangent vector of $T_p \mathcal{M}$ that is dual to the derivative $\mathrm{d}_p f: T_p \mathcal{M} \to \mathbb{R}$ under the identification of linear forms and vectors via the inner product $\mathbf{g}_p$ on $T_p \mathcal{M}$. The first-order necessary condition for a function $f : \mathcal{M} \to \mathbb{R}$ to have a minimum at $p$ is that the gradient $(\nabla f)(p) = 0$; see, e.g., \cite{manoptAb,mnoptbo}.
	
		The \emph{covariant derivative} allows vector fields to be differentiated. Intuitively, this can be understood as follows. Let $p\in\mathcal{M}$ be a point and $X, Y$ smooth vector fields. There is a unique smooth curve $\gamma(t)$ with $\gamma(0)=p$ and $\gamma'(0)=X(p)$, obtained by integrating the flow defined by $X$ from $p$. The covariant derivative $\nabla_Y X$ is then measuring at $p$ how the initial tangent vector $X(p)$ is moving if $p$ is moved infinitesimally in the direction of $Y(p)$. There is a preferred covariant derivative for Riemannian manifolds.
		This \emph{Levi--Civita connection} $\nabla : T\mathcal{M} \times T\mathcal{M} \to T\mathcal{M}$ is the unique covariant derivative that is torsion-free and compatible with the metric in the sense of \cite[Theorem 2.2.2]{Petersen}.
	
		The connection $\nabla$ applied to the gradient $\nabla f$ results in the \emph{Riemannian Hessian} $\nabla^2 f$ of $f$ by \cite[Proposition 2.2.6]{Petersen}. Evaluated at $p$, this object can be viewed as a self-adjoint linear endomorphism on $T_p \mathcal{M}$, i.e., $(\nabla^2 f)(p) : T_p\mathcal{M} \to T_p \mathcal{M}$. Given two self-adjoint linear endomorphisms $A, B$, we write $A \succcurlyeq B$ (resp., $A \succ B$) to mean that the eigenvalues of $A - B$ are all positive (resp., strictly positive). In particular, $A \succcurlyeq 0$ (resp., $A \succ 0$) means that $A$ is positive semidefinite (resp., positive definite). The second-order necessary condition for a function $f$ to have a minimum at $p$ is that $(\nabla^2 f)(p) \succcurlyeq 0$; see, e.g., \cite{manoptAb,mnoptbo}.
	
		The Nash--Moser implicit function theorem is an advanced but standard result in differential geometry; see, e.g., Hamilton \cite{Hamilton1982} for an overview---Theorem 3.3.4 specifically is a suitable version for our purpose. An accessible version for Riemannian manifolds was presented in full detail by Bergmann and Zimmermann \cite{Zimmermann2024multihermi}. In our notation, the latter can be phrased as follows.
	
		\begin{theorem}[{\protect Essentially \cite[Theorem 6]{Zimmermann2024multihermi}}] \label{thm_implicit_function}
	     Let $(\mathcal{M},\mathbf{g})$ be a Riemannian manifold of dimension $d$, and let $G: \mathcal{M} \times \mathbb{R}^n \rightarrow T \mathcal{M}$ be a smooth vector field on an open subset $\mathcal{U} \subset \mathcal{M}\times\mathbb{R}^n$.
	     If there is a $\left(p^*, x^*\right) \in \mathcal{U}$ such that
	     \begin{enumerate}
	      \item $G\left(p^*, x^*\right)= 0 \in T_{p^*}\mathcal{M}$, and
	      \item the covariant derivative $\nabla G^*$ at $p$ of the smooth vector field $G_{x^*} := G|_{\mathcal{M},\{x^*\}}$ has rank $d$,
	     \end{enumerate}
		 then there exists an open neighborhood $\Omega \subset \mathbb{R}^d$ of $x^*$ and a smooth map $\widehat{q}: \Omega \rightarrow \mathcal{M}$ that solves the implicit equation $G(\widehat{q}(x),x) = 0$.
	     \end{theorem}
	
	     We now have the necessary background to prove \cref{thm_mtsm_smooth}.
	
		\begin{proof}[Proof of \cref{thm_mtsm_smooth}]
		Consider the map
		\[
		 G : \mathcal{M} \times \Omega \longrightarrow T\mathcal{M}, \quad
		 (p,x) \longmapsto \frac{1}{2} \sum_{j=1}^R \varphi_{j}( \widehat{f}_j(x) ) \cdot \nabla \dist_{\mathcal{M}}(p, \widehat{f}_j(x))^2,
		\]
		which for a fixed $x$ takes $(p,x)$ to the gradient (in the first variable) of the objective function inside the $\argmin$ in \cref{eqn_mtsm}, scaled by $\frac{1}{2}$. Herein, $\nabla \dist_{\mathcal{M}}(p, \widehat{y}_j)^2$ is the gradient of the squared distance function from $p$ to the fixed point $\widehat{y}_j := \widehat{f}_j(x)$. As $G$ is a linear combination of gradient fields with smooth functions as coefficients, it follows that $G$ is a smooth vector field.
	
		If follows from \cref{thm_rtsm_unique} that for every $x\in\Omega$ there exists a global minimizer $q(x)$ of the minimization problem in \cref{eqn_mtsm}. The first-order necessary condition for optimality implies that $G(q(x),x) = 0$.
		Let $r_j(p) = \dist_{\mathcal{M}}(p, \widehat{y}_j)$.
		Fixing $x$, we see that the covariant derivative of $G_x = G|_{\mathcal{M},x}$ satisfies
		\[
		 \nabla G_x = \sum_{j=1}^R \varphi_{j}( \widehat{y}_j ) \cdot \nabla^2 \frac{1}{2} r_j^2.
		\]
		If we could prove that $\mathrm{rank}(\nabla G_x)=m$ at $q(x)$, then the proof would be concluded by applying \cref{thm_implicit_function}.
		Since the $\varphi_j$'s are a convex combination, it suffices to show that the Riemannian Hessians $\nabla^2 \frac{1}{2} r_j^2 \succ 0$ if $\varphi_j(y_j) > 0$.
	
		A close inspection of Karcher's result \cite[Theorem 1.2]{karcher} reveals that this is the case under the assumptions of \cref{thm_rtsm_unique}.
		Indeed, Karcher shows in equation (1.2.2) in \cite{karcher} that in manifolds with all sectional curvatures negative, we have
		\(
		 \nabla^2 \frac{1}{2} r_j^2 \succcurlyeq \mathbf{g} \succ 0,
		\)
		where $\mathbf{g}$ is the Riemannian metric; this is \cite[Lemma 6.2.5]{Petersen}. Hence, the claim holds when $K=0$ in \cref{thm_rtsm_unique}.
		For manifolds with positive sectional curvatures bounded by $K > 0$, a much more advanced argument is required to obtain (1.2.3) in \cite{karcher}. Essentially, Karcher exploits $\nabla^2 \frac{1}{2} r_j^2 = \mathrm{d} r_j^2 + r_j \nabla^2 r_j$ (by \cite[Exercise 2.5.6]{Petersen}) and Rauch's comparison theorem \cite[Theorem 6.4.3]{Petersen} applied to $\nabla^2 r_j$ to yield
		\(
		 \nabla^2 \frac{1}{2} r_j^2 \succcurlyeq \sqrt{K} \mathrm{cotan}(\sqrt{K} r_j) \mathbf{g}.
		\)
		Since the cotangent is strictly positive for $0 \le r_j(p) < \frac{\pi}{2}$, it follows that $\nabla^2 \frac{1}{2} r_j^2 \succ 0$ if $r_j < \frac{\pi}{2 \sqrt{K}}$. As in the proof of C3 in \cref{thm_rtsm_unique}, assumption A3 implies that the distance from $f(x)$ to each $\widehat{y}_j$ is bounded by $5 \epsilon$. By \cite[Theorem 1.2]{karcher}, the minimum $q(x)$ of \cref{eqn_mtsm} will be attained in the interior of the geodesic ball of radius $5\epsilon$ centered at $f(x)$. The maximum distance $r_j(q(x))$ from $q(x)$ to the $\widehat{y}_j$'s is then $10\epsilon < \min\{ \mathrm{inj}(p_j^*), \frac{\pi}{2 \sqrt{k}}\} \le \frac{\pi}{2 \sqrt{K}}$ by the triangle inequality. This shows that $\nabla \frac{1}{2} r_j^2(q(x)) \succ 0$ for all $j$ such that $\varphi_j(\widehat{y}_j) > 0$ and concludes the proof.
		\end{proof}

\bibliographystyle{siamplain}
\bibliography{reference_new}

\end{document}